\documentclass[a4paper,12pt,reqno]{amsart}
\usepackage{amsmath}
\usepackage{amsfonts}
\usepackage{amssymb}
\usepackage{amsthm}
\usepackage{amscd}
\usepackage{color}
\usepackage{latexsym}
\usepackage{graphics}

\newcommand{\la}{\langle}
\newcommand{\ra}{\rangle}
\newcommand{\x}{\times}
\newcommand{\ox}{\otimes}
\newcommand{\too}{\longrightarrow}

\newcommand{\bd}{\partial}

\newcommand{\SA}{{\mathcal{A}}}
\newcommand{\SC}{{\mathcal{C}}}
\newcommand{\SH}{{\mathcal{H}}}
\newcommand{\SM}{{\mathcal{M}}}

\newcommand{\ZZ}{\mathbb{Z}}
\newcommand{\CC}{\mathbb{C}}
\newcommand{\RR}{\mathbb{R}}

\newcommand{\im}{\operatorname{im}}
\newcommand{\GL}{\operatorname{GL}}
\renewcommand{\O}{\operatorname{O}}

\DeclareMathOperator{\Hom}{Hom}
\DeclareMathOperator{\Id}{Id}
\DeclareMathOperator{\Diff}{Diff}
\newcommand{\inc}{\hookrightarrow}

\numberwithin{equation}{section}

\newtheorem{proposition}{Proposition}[section]
\newtheorem{theorem}[proposition]{Theorem}
\newtheorem{lemma}[proposition]{Lemma}
\newtheorem{corollary}[proposition]{Corollary}
\theoremstyle{definition}
\newtheorem{definition}[proposition]{Definition}

\theoremstyle{remark}
\newtheorem{remark}[proposition]{Remark}

\begin{document}

\title[Homotopic properties of K\"ahler orbifolds]{Homotopic properties of K\"ahler 
orbifolds}

\author[G. Bazzoni]{Giovanni Bazzoni}

\address{Philipps-Universit\"at Marburg, FB Mathematik \& Informatik,
Hans-Meerwein-Str. 6 Campus Lahnberge, 35032 Marburg, Germany}
\email{bazzoni@mathematik.uni-marburg.de}

\author[I. Biswas]{Indranil Biswas}

\address{School of Mathematics, Tata Institute of Fundamental
Research, Homi Bhabha Road, Bombay 400005, India}
\email{indranil@math.tifr.res.in}

\author[M. Fern\'andez]{Marisa Fern\'{a}ndez}

\address{Universidad del Pa\'{\i}s Vasco,
Facultad de Ciencia y Tecnolog\'{\i}a, Departamento de Matem\'aticas,
Apartado 644, 48080 Bilbao, Spain}
\email{marisa.fernandez@ehu.es}

\author[V. Mu\~{n}oz]{Vicente Mu\~{n}oz}

\address{Facultad de Ciencias Matem\'aticas, Universidad
Complutense de Madrid, Plaza de Ciencias
3, 28040 Madrid, Spain}
\email{vicente.munoz@mat.ucm.es}

\author[A. Tralle]{Aleksy Tralle}

\address{Department of Mathematics and Computer Science, University of Warmia
and Mazury, S\l\/oneczna 54, 10-710, Olsztyn, Poland}
\email{tralle@matman.uwm.edu.pl}



\dedicatory{To Simon Salamon on the occasion of his 60th birthday}

\begin{abstract}
We prove the formality and the evenness of odd-degree Betti numbers for compact 
K\"ahler orbifolds, by adapting the classical proofs for K\"ahler manifolds. As a 
consequence, we obtain examples of symplectic orbifolds not admitting any K\"ahler 
orbifold structure. We also review the known examples of non-formal simply 
connected Sasakian manifolds, and produce an example of a non-formal quasi-regular Sasakian manifold
with Betti numbers $b_1=0$ and $b_2\,> 1$.
\end{abstract}

\maketitle

\section{Introduction}\label{sec:intro}

A K\"ahler manifold $M$ is a complex manifold, admitting 
a Hermitian metric $h$, such that the $(1,1)$-form $\omega=\mathrm{Im}\,h$ is closed, 
and so symplectic, where $\mathrm{Im}\,h$ is the imaginary part of $h$. The
real part $g=\mathrm{Re}\,h$ of $h$ is a Riemannian metric
which is is called the \textit{K\"ahler metric} associated to $\omega$. 
If a compact manifold admits a K\"ahler metric, then it inherits some very striking topological properties,
for example: theory of K\"ahler groups, evenness of odd-degree Betti numbers, hard Lefschetz theorem,
formality of the rational homotopy type (see \cite{DGMS,Wells}).

 K\"ahler metrics can be also defined on {\em orbifolds}.
A smooth orbifold $X$, of dimension $n$, is a Hausdorff topological space admitting 
an open cover $\{U_i\}_{i\in I}$, such that each $U_i$ is homeomorphic
to a quotient $\Gamma_{i}{\backslash}\widetilde{U}_i$, where $\widetilde{U}_i \subset {\mathbb{R}}^n$ is an open subset,
$\Gamma_i\subset \GL(n, \mathbb{R})$ a f{}inite group acting on $\widetilde{U}_i $, and
there is a $\Gamma_i$-invariant 
continuous map $\varphi_i\,\colon\,\widetilde{U}_i\,\longrightarrow\,U_i$ inducing a homeomorphism from 
$\Gamma_{i}{\backslash}\widetilde{U}_i$ onto $U_i$. 
Moreover, the gluing maps are
required to be smooth and compatible with the group action (see Section \ref{sec:orbifolds} for the details).

The orbifold differential forms on a smooth orbifold are defined
in local charts as $\Gamma_{i}$-invariant differential forms
on each open set $\widetilde{U}_i$, which are compatible with the gluing maps.
The de Rham complex is defined
in the same way as for smooth manifolds, and the de Rham cohomology is equal to the singular cohomology. 
This result and Poincar\'e duality theorem were first proved by Satake, who introduced the notion of
orbifold under the name ``$V$-manifold" \cite{S1}. Since Satake, 
various index theorems 
were generalized by Kawasaki to the category of $V$-manifolds
(see \cite{Ka1, Ka2, Ka3} and the book by Atiyah \cite{Atiyah}). In the late 1970s, Thurston \cite{Thurston} rediscovered 
the concept of $V$-manifold, under the name of orbifold, in his study of the geometry
of 3-manifolds, and defined the orbifold fundamental group.
Even though orbifolds were already very important objects in mathematics, with the work of Dixon, Harvey, Vafa and Witten on conformal
field theory  \cite{DHVW}, the interest on orbifolds dramatically increased, due to their role in string theory
(see \cite{Adem} and the references therein).
A complex orbifold, of complex dimension $n$, is an orbifold $X$ with charts 
$(U_i,\, \widetilde{U}_i, \,\Gamma_{i}, \,\varphi_i)$ as above satisfying the
conditions that
$\widetilde{U}_i \subset {\mathbb{C}}^n$, $\Gamma_i\subset \GL(n, \mathbb{C})$, and all
the gluing maps are given by biholomorphisms. 
One can also define {\em orbifold complex forms} and {\em orbifold
Hermitian metrics} on $X$ (see Section \ref{sec:kahler-orb} for the details). 
A complex orbifold $X$ is said to be {\em K\"ahler} if $X$ admits
an orbifold Hermitian metric such that the associated orbifold K\"ahler form is closed. 
The notion of complex orbifold
was introduced, under the name of {\em complex $V$-manifold}, by Baily \cite{Baily1} who
generalized the Hodge decomposition theorem to Riemannian $V$-manifolds. 

Although compact K\"ahler orbifolds are not smooth manifolds in general, they
continue to possess some topological properties of K\"ahler manifolds. There
are two possible points of view to look at topological properties of orbifolds. One is to look at the 
topological properties of the underlying topological space, and the other is to look at specific
orbifold invariants such as the orbifold fundamental group or the orbifold cohomology. We shall
focus on the former, since the latter is more adequate for the interplay between the topological space
and the subspaces defining the orbifold ramification locus. So when we talk of the fundamental group
or the homology or cohomology of the orbifold, we refer to those of the underlying topological spaces. 

A compact K\"ahler orbifold is the leaf space of a foliation on a
compact manifold $Y$ \cite[Proposition 4.1]{GHS}, and
such a foliation is transversely K\"ahler \cite[Proposition 1.4]{Wang-Zaffran}. Moreover, the basic cohomology of $Y$ is isomorphic to the singular cohomology of the
orbifold over $\mathbb{C}$ \cite[5.3]{Pflaum}.
In \cite{Wang-Zaffran} it is proved that any compact K\"ahler orbifold satisfies the
hard Lefschetz property. This is done
by using a result of El Kacimi-Alaoui \cite{El Kacimi} which says that the basic cohomology
of a transversely K\"ahler foliation on a compact manifold satisfies the hard Lefschetz property.
On the other hand, the $dd^{c}$-lemma for the algebra of the basic forms 
of a transversely K\"ahler foliation was shown in \cite{CW}. Also
in \cite{El Kacimi} it is proved that the basic Dolbeault cohomology of a transversely K\"ahler foliation on a compact manifold
has the same properties as the Dolbeault cohomology of a compact K\"ahler manifold. 
So, compact K\"ahler orbifolds possess
the earlier mentioned topological properties of K\"ahler manifolds. Regarding the
fundamental group of
a K\"ahler orbifold, the fundamental group of the topological space underlying the
orbifold actually coincides with the fundamental group of a
resolution {\cite[Theorem 7.8.1]{Kol}}. 
Therefore, these fundamental groups of K\"ahler orbifolds
satisfy the same restrictions as the fundamental groups of compact K\"ahler manifolds.

The main purpose of this paper is to prove that compact K\"ahler orbifolds are formal.
This is achieved by adapting the proof of formality for K\"ahler manifolds
given in \cite{DGMS}. The machinery used is described in Sections \ref{sec:formality}, \ref{sec:orbifolds} and \ref{sec:ellipticoperators}.
In Sections \ref{sec:formality} and \ref{sec:orbifolds} we focus on the formality of smooth manifolds and orbifolds, respectively, and in
Section \ref{sec:ellipticoperators} we study elliptic operators on complex orbifolds
following \cite{Wells}, but it was first developed by Baily in the aforementioned paper \cite{Baily1}.
Then, in Section \ref{sec:kahler-orb} the orbifold Dolbeault cohomology
of a complex orbifold is defined, and the $\bd{\overline\bd}$-lemma for compact
K\"ahler orbifolds is proved (Lemma \ref{lem:dd-lemma}). The formality of
compact K\"ahler orbifolds is deduced using this
(Theorem \ref{thm:K-orb-formal}). 
Moreover, in Proposition \ref{prop:3} we prove that the orbifold Dolbeault cohomology 
is equipped with an analogue of the Hodge decomposition
for K\"ahler manifolds. Consequently, the odd Betti numbers of compact K\"ahler orbifolds are even.
(A Hodge decomposition theory for nearly K\"ahler manifolds was developed by Verbitsky in \cite{MishaV}, where it is 
noted that this theory works also for nearly K\"ahler orbifolds.)
In Section \ref{so:nonko}, we produce examples
of symplectic orbifolds which do not admit any K\"ahler orbifold metric (as
they are non-formal or they do not possess the hard Lefschetz property).

Closely related to K\"ahler orbifolds are Sasakian manifolds. Such a manifold is a Riemannian manifold
$(N, g)$, of dimension $2n+1$, such that its cone
$(N\times{\mathbb{R}}^+, \,g^c = t^2 g\,+\,dt^2)$ is K\"ahler,
and so the holonomy group for $g^c$ is a subgroup of
$\mathrm{U}(n+1)$. 
The K\"ahler structure on the cone induces a Sasakian
structure on the base of the cone. In particular, the
complex structure on the cone gives rise to a Reeb vector field. 

If $N$ admits a Sasakian structure, then in \cite{OV} it is proved that
$N$ also admits a quasi-regular Sasakian structure. The space $X$ of leaves of
a quasi-regular Sasakian structure is a K\"ahler
orbifold with cyclic quotient singularities, and there is an orbifold circle bundle
$S^1 \,\hookrightarrow\, N\,\stackrel{\pi}{\longrightarrow}\, X$ such that the contact
form $\eta$ satisfies the equation $d\eta\,=\,\pi^*\omega$,
where $\omega$ is the orbifold K\"ahler form. 
If $X$ is a K\"ahler manifold, then the Sasakian structure on $N$ is regular.

However, opposed to K\"ahler orbifolds, formality is not an obstruction to the 
existence of a Sasakian structure even on simply connected manifolds \cite{BFMT}. On 
the other hand, all quadruple and higher order Massey products are trivial on any 
Sasakian manifold. In fact, in \cite{BFMT} it is proved that, for any $n\,\geq 3$, 
there exists a simply connected compact regular Sasakian manifold, of dimension 
$2n+1$, which is non-formal, in fact not 3-formal, in the sense of Definition 
\ref{def:primera}. (Note that simply connected compact manifolds of dimension at 
most $6$ are formal \cite{{FM},{N-Miller}}.) In Section \ref{non-formal-sasakian} we 
review these examples and show that they have a 
non-trivial (triple) Massey product, which implies that they are not formal.

Regarding the simply connected compact regular Sasakian manifolds that are formal, 
the odd-dimensional sphere $S^{2n+1}$ is the most basic example of them. 
By Theorem \ref{fm2:criterio2} we know that any 7-dimensional simply connected compact manifold (Sasakian or not)
with $b_2 \,\leq 1$ is formal. 
In \cite{FIM}, examples are given of simply connected formal compact regular Sasakian manifolds, of dimension $7$, with second Betti number $b_2\,\geq 2$.
This result and Proposition \ref{sasak-formal:b2=1} (Section \ref{non-formal-sasakian}) show that, for every $n\,\geq 3$, there exists a simply connected compact regular Sasakian manifold, of 
dimension $2n+1\,\geq 7$, which is formal and has $b_2\,\not=\,0$. 
We end up with an example of a quasi-regular (non-regular)
Sasakian manifold with $b_1=0$ which is non-formal. 
 
%%%%%%%%%%%%%%%%%%%%%%%%%%%%%%%%%%%%%%%
\section{ Formality of manifolds}\label{sec:formality}
%%%%%%%%%%%%%%%%%%%%%%%%%%%%%%%%%

In this section some definitions and results about minimal models
and Massey products { on smooth manifolds} are reviewed; see \cite{DGMS, FHT} for more details. 

We work with the {\em differential graded commutative algebras}, or DGAs,
over the field $\mathbb R$ of real numbers. The degree of an
element $a$ of a DGA is denoted by $|a|$. A DGA $(\mathcal{A},\,d)$ is {\it minimal\/}
if:
\begin{enumerate}
 \item $\mathcal{A}$ is free as an algebra, that is $\mathcal{A}$ is the free
 algebra $\bigwedge V$ over a graded vector space $V\,=\,\bigoplus_i V^i$, and

 \item there is a collection of generators $\{a_\tau\}_{\tau\in I}$
indexed by some well ordered set $I$, such that
 $|a_\mu|\,\leq\, |a_\tau|$ if $\mu \,< \,\tau$ and each $d
 a_\tau$ is expressed in terms of the previous $a_\mu$, $\mu\,<\,\tau$.
 This implies that $da_\tau$ does not have a linear part.
 \end{enumerate}

In our context, the main example of DGA is the de Rham complex $(\Omega^*(M),\,d)$
of a smooth manifold $M$, where $d$ is the exterior differential.

The cohomology of a differential graded commutative algebra $(\mathcal{A},\,d)$ is
denoted by $H^*(\mathcal{A})$. $H^*(\mathcal{A})$ is naturally a DGA with the 
product inherited from that on $\mathcal{A}$ while the differential on
$H^*(\mathcal{A})$ is identically zero.

A DGA $(\mathcal{A},\,d)$ is called {\it connected} if 
$H^0(\mathcal{A})\,=\,\RR$, and it is called {\em $1$-connected\/} if, in 
addition, $H^1(\mathcal{A})\,=\,0$.

Morphisms between DGAs are required to preserve the degree and to commute with the 
differential. We shall say that $(\bigwedge V,\,d)$ is a {\it minimal model} of a 
differential graded commutative algebra $(\mathcal{A},\,d)$ if $(\bigwedge V,\,d)$ 
is minimal and there exists a morphism of differential graded algebras 
$$\rho\,\colon\, {(\bigwedge V,\,d)}\,\longrightarrow\, {(\mathcal{A},\,d)}$$ 
inducing an isomorphism $\rho^*\,\colon\, H^*(\bigwedge 
V)\,\stackrel{\sim}{\longrightarrow}\, H^*(\mathcal{A})$ of cohomologies. 
{ In~\cite{Halperin}, Halperin proved that any connected differential graded algebra 
$(\mathcal{A},\,d)$ has a minimal model unique up to isomorphism.
For $1$--connected differential algebras, a similar result was proved by Deligne, Griffiths,
Morgan and Sullivan~\cite{DGMS,GM,Su}.}

A {\it minimal model\/} of a connected smooth manifold $M$
is a minimal model $(\bigwedge V,\,d)$ for the de Rham complex
$(\Omega^*(M),\,d)$ of differential forms on $M$. If $M$ is a simply
connected manifold, then the dual of the real homotopy vector
space $\pi_i(M)\otimes \RR$ is isomorphic to the space $V^i$ of generators in degree $i$, for any $i$.
The latter also happens when $i\,>\,1$ and $M$ is nilpotent, that
is, the fundamental group $\pi_1(M)$ is nilpotent and its action
on $\pi_j(M)$ is nilpotent for all $j\,>\,1$ (see~\cite{DGMS}).

We say that a DGA $(\mathcal{A},\,d)$ is a {\it model} of a manifold $M$
if $(\mathcal{A},\,d)$ and $M$ have the same minimal model. Thus, if $(\bigwedge V,\,d)$ is the minimal
model of $M$, we have 
$$
 (\mathcal{A},\,d)\, \stackrel{\nu}\longleftarrow\, {(\bigwedge V,\, d)}\, \stackrel{\rho}\longrightarrow\, (\Omega^{*}(M),\,d),
$$
where $\rho$ and $\nu$ are quasi-isomorphisms, meaning morphisms of DGAs such that the
induced homomorphisms in cohomology are isomorphisms.

Recall that a minimal algebra $(\bigwedge V,\,d)$ is called
{\it formal} if there exists a
morphism of differential algebras $\psi\,\colon\, {(\bigwedge V,\,d)}\,\longrightarrow\,
(H^*(\bigwedge V),\,0)$ inducing the identity map on cohomology.
A DGA $(\mathcal{A},d)$ is formal if its minimal model is formal.

A smooth manifold $M$ is called {\it formal\/} if its minimal model is
formal. Many examples of formal manifolds are known: spheres, projective
spaces, compact Lie groups, symmetric spaces, flag manifolds,
and compact K\"ahler manifolds.

The formality property of a minimal algebra is characterized as follows.

\begin{proposition}[\cite{DGMS}]\label{prop:criterio1}
A minimal algebra $(\bigwedge V,\,d)$ is formal if and only if the space $V$
can be decomposed into a direct sum $V\,=\, C\oplus N$ with $d(C) \,=\, 0$
and $d$ injective on $N$, such that every closed element in the ideal
$I(N)$ in $\bigwedge V$ generated by $N$ is exact.
\end{proposition}

This characterization of formality can be weakened using the concept of
$s$-formality introduced in \cite{FM}.

\begin{definition}\label{def:primera}
A minimal algebra $(\bigwedge V,\,d)$ is $s$-formal
($s\,>\, 0$) if for each $i\,\leq\, s$
the space $V^i$ of generators of degree $i$ decomposes as a direct
sum $V^i\,=\,C^i\oplus N^i$, where the spaces $C^i$ and $N^i$ satisfy
the following conditions:
\begin{enumerate}

\item $d(C^i) = 0$,

\item the differential map $d\,\colon\, N^i\,\longrightarrow\, \bigwedge V$ is
injective, and

\item any closed element in the ideal
$I_s=I(\bigoplus\limits_{i\leq s} N^i)$, generated by the space
$\bigoplus\limits_{i\leq s} N^i$ in the free algebra $\bigwedge
(\bigoplus\limits_{i\leq s} V^i)$, is exact in $\bigwedge V$.
\end{enumerate}
\end{definition}

A smooth manifold $M$ is $s$-formal if its minimal model
is $s$-formal. Clearly, if $M$ is formal then $M$ is $s$-formal for every $s\,>\,0$.
The main result of \cite{FM} shows that sometimes the weaker
condition of $s$-formality implies formality.

\begin{theorem}[\cite{FM}]\label{fm2:criterio2}
Let $M$ be a connected and orientable compact differentiable
manifold of dimension $2n$ or $(2n-1)$. Then $M$ is formal if and
only if it is $(n-1)$-formal.
\end{theorem}

One can check that any simply connected compact manifold 
is $2$-formal. Therefore, Theorem \ref{fm2:criterio2} implies that any
simply connected compact manifold of dimension at most six
is formal. (This result was proved earlier in \cite{N-Miller}.)

In order to detect non-formality, instead of computing the minimal
model, which is usually a lengthy process, one can use Massey
products, which are obstructions to formality. The simplest type
of Massey product is the triple (also known as ordinary) Massey
product. This will be defined next.

Let $(\mathcal{A},\,d)$ be a DGA (in particular, it can be the de Rham complex
of differential forms on a smooth manifold). Suppose that there are
cohomology classes $[a_i]\,\in\, H^{p_i}(\mathcal{A})$, $p_i\,>\,0$,
$1\,\leq\, i\,\leq\, 3$, such that $a_1\cdot a_2$ and $a_2\cdot a_3$ are
exact. Write $a_1\cdot a_2=da_{1,2}$ and $a_2\cdot a_3=da_{2,3}$.
The {\it (triple) Massey product} of the classes $[a_i]$ is defined as
$$
\langle [a_1],[a_2],[a_3] \rangle \,=\, 
[ a_1 \cdot a_{2,3}+(-1)^{p_{1}+1} a_{1,2}
\cdot a_3] \in \frac{H^{p_{1}+p_{2}+ p_{3} -1}(\mathcal{A})}{[a_1]\cdot H^{p_{2}+ p_{3} -1}(\mathcal{A})+[a_3]\cdot H^{p_{1}+ p_{2} -1}(\mathcal{A})}.
$$

Note that a Massey product $\langle [a_1],[a_2],[a_3] \rangle$ on $(\mathcal{A},\,d_{\mathcal{A}})$
is zero (or trivial) if and only if there exist $\widetilde{x}, \widetilde{y}\in \mathcal{A}$ such that
$a_1\cdot a_2=d_{\mathcal{A}}\widetilde{x}$, \, $a_2\cdot a_3=d_{\mathcal{A}}\widetilde{y}$\, and
$$0\,=\,[ a_1 \cdot \widetilde{y}+(-1)^{p_{1}+1}\widetilde{x}\cdot a_3]
\,\in\, H^{p_{1}+p_{2}+ p_{3} -1}(\mathcal{A})\, .$$

We will use also the following property.

\begin{lemma} \label{lemm:massey-models}
Let $M$ be a connected smooth manifold. Then, Massey products on $M$
can be calculated by using any model of $M$.
\end{lemma}

\begin{proof}
It is enough to prove the following: $\varphi\,\colon\,(\mathcal{A},\,d_{\mathcal{A}})
\,\longrightarrow\, (\mathcal{B},\,d_{\mathcal{B}})$
is a quasi-isomorphism, then $$\varphi^*(\langle [a_1],[a_2],[a_3] \rangle)
\,=\,\langle [a_1'],[a_2'],[a_3'] \rangle$$ for
$[a_j']=\varphi^*([a_j])$. But this is clear; indeed, take
$a_1\cdot a_2\,=\,d_{\mathcal{A}}{x}$, \, $a_2\cdot a_3\,=\,d_{\mathcal{A}}{y}$\, and let
 $$
 f\,=\,[ a_1 \cdot {y}+(-1)^{p_{1}+1} {x}\cdot a_3]\,\in\,
 \frac{H^{p_{1}+p_{2}+ p_{3} -1}(\mathcal{A})}{[a_1]\cdot
 H^{p_{2}+ p_{3} -1}(\mathcal{A})+[a_3]\cdot H^{p_{1}+ p_{2} -1}(\mathcal{A})}\,
 $$
be its Massey product $\langle [a_1],[a_2],[a_3] \rangle$. Then the elements $a_j'\,=\,
\varphi(a_j)$ satisfy
$a_1'\cdot a_2'\,=\,d_{\mathcal{B}}{x'}$, \, $a_2'\cdot a_3'\,=\,d_{\mathcal{B}}{y'}$,
where $x'\,=\,\varphi(x)$, $y'\,=\,\varphi(y)$. Therefore,
 $$
 f'\,=\,[ a_1' \cdot {y'}+(-1)^{p_{1}+1} {x'}\cdot a_3'] \,=\,\varphi^*(f) \,\in\,
 \frac{H^{p_{1}+p_{2}+ p_{3} -1}(\mathcal{B})}{[a_1']\cdot
 H^{p_{2}+ p_{3} -1}(\mathcal{B})+[a_3']\cdot H^{p_{1}+ p_{2} -1}(\mathcal{B})}\,
 $$
is the Massey product $\langle [a_1'],[a_2'],[a_3'] \rangle$.
 \end{proof}

Now we move to the definition of higher Massey products
(see~\cite{TO}). Given $$[a_{i}]\,\in\, H^*(\mathcal{A}),\ \ 1\,\leq\, i
\,\leq\, t,\ \ t\geq 3\, ,$$
the Massey product $\la [a_{1}],[a_{2}],\cdots,[a_{t}]\ra$, is defined
if there are elements $a_{i,j}$ on $\mathcal{A}$, with $1\,\leq\, i\,\leq \,j\,\leq\,
t$ and $(i,\, j)\,\not=\, (1,\, t)$, such that
 \begin{equation}\label{eqn:gm}
 \begin{aligned}
 a_{i,i}&= a_i\, ,\\
 d\,a_{i,j}&= \sum\limits_{k=i}^{j-1} (-1)^{|a_{i,k}|}a_{i,k}\cdot
 a_{k+1,j}\, .
 \end{aligned}
 \end{equation}
Then the {\it Massey product} is the set of all cohomology classes
 $$
 \la [a_{1}],[a_{2}],\cdots,[a_{t}] \ra
$$
$$
=\, \left\{
 \left[\sum\limits_{k=1}^{t-1} (-1)^{|a_{1,k}|}a_{1,k} \cdot
 a_{k+1,t}\right] \ | \ a_{i,j} \mbox{ as in (\ref{eqn:gm})}\right\}
 \subset H^{|a_{1}|+ \cdots +|a_{t}|
 -(t-2)}(\mathcal{A})\, .
 $$
We say that the Massey product is {\it zero} if
$$0\,\in\, \la [a_{1}],[a_{2}],\cdots,[a_{t}] \ra\, .$$ 
Note that the higher order Massey product $\la [a_1],[a_2],\cdots,[a_t]\ra$ of order $t\geq 4$ is defined if
all the Massey products $\la [a_i],\cdots,[a_{i+p-1}]\ra$ of order $p$, where $3\leq p\leq t-1$ and $1 \leq i \leq t-p+1$,
are defined and trivial. 

Massey products are related to formality by the
following well-known result.

\begin{theorem}[\cite{DGMS,TO}] \label{theo:Massey products}
A DGA which has a non-zero Massey product is not formal.
\end{theorem}

Another obstruction to the formality is given by the 
$a$-{\it Massey products} introduced in \cite{CFM}, which
generalize the triple Massey product and have the advantage of
being simpler to compute compared to the higher order Massey
products. They are defined as follows. Let $(\mathcal{A},\,d)$ be a DGA, and let
$a, b_1, \ldots, b_n \,\in \,\mathcal{A}$ be
closed elements such that the degree $|a|$ of $a$ is even and $a\cdot b_i$ is
exact for all $i$. Let $\xi_i$ be any form such that $d\xi_i \,=\, a
\cdot b_i$. Then the {\it $n^{th}$ order $a$-Massey
product} of the $b_i$ is the subset
 $$
\la a; b_1,\ldots,b_n \ra
$$
$$
:= \left\{ \left[\sum_i (-1)^{|\xi_1|+\cdots+ |\xi_{i-1}|}
{\xi_1}\cdot\ldots \cdot{\xi_{i-1}} \cdot b_i
\cdot \xi_{i+1}\cdot \ldots\cdot \xi_n\right] \, |\,
 d\xi_i = a\cdot b_i\right\}\subset H^*(\mathcal{A})\, .
$$
We say that the $a$-Massey product is {\it zero} if $0\,\in \,\la
a; b_1,\ldots,b_n\ra$. 

\begin{theorem}[\cite{CFM}] \label{theo:amassey and formality}
A DGA which has a non-zero $a$-Massey product is not
formal.
\end{theorem}

%%%%%%%%%%%%%%%%%%%%%%%%%%%%%%%%%%%%%%%
\section{Orbifolds}\label{sec:orbifolds}
%%%%%%%%%%%%%%%%%%%%%%%%%%%%%%%%%

{ In this section, we collect some results about smooth orbifolds and formality of these spaces (see \cite{Adem, BG, GM, Kleiner-Lott, S1, S2, Thurston}).

Let $X$ be a topological space. Fix an integer $n\,>\,0$. 
An {\em orbifold chart} $(U, {\widetilde U}, \Gamma, \varphi)$ on $X$ 
consists of an open set $U \,\subset\, X$, a connected and open set
${\widetilde U} \,\subset\, \RR^n$, a f{}inite group $\Gamma\subset \GL(n, \mathbb{R})$ acting smoothly and effectively 
on $\widetilde{U}$, and a continuous map 
$$
\varphi\,\colon\,\widetilde{U}\,\longrightarrow\,U,
$$
 which is  $\Gamma$-invariant 
(that is $\varphi\,=\,\varphi\circ\gamma$, for all $\gamma \in\Gamma$) and such that it induces a homeomorphism 
$$\Gamma\backslash{\widetilde U} \stackrel{\cong}{\longrightarrow} U$$
from the quotient space $\Gamma\backslash{\widetilde U}$ onto $U$.

\begin{definition}\label{def:orbifold}
A smooth orbifold $X$, of dimension $n$, is a Hausdorff, paracompact topological 
 space endowed with an {\em orbifold atlas} $\SA=\{(U_i, {\widetilde U}_i, \Gamma_i, \varphi_i)\}_{i\,\in\,I}$, that is $\SA$ is a family of orbifold charts
which satisfy the following conditions:
 \begin{enumerate}
 \item[i)] $\{U_i\}_{i\,\in\,I}$ is an open cover of $X$;
 \item[ii)] If $(U_i, {\widetilde U}_i, \Gamma_i, \varphi_i)$ and $(U_j, {\widetilde U}_j, \Gamma_j, \varphi_j)$, $i, j\in I$,
 are two orbifold charts, with $U_i\,\cap\,U_j\,\not=\emptyset$, then for each point $p\,\in\,U_i\,\cap\,U_j$ there exists
an orbifold chart $(U_k, {\widetilde U}_k, \Gamma_k, \varphi_k)$ $(k\,\in\, I)$ such that $p\,\in\,U_k\,\subset\,U_i\,\cap\,U_j$;
 \item[iii)] If $(U_i, {\widetilde U}_i, \Gamma_i, \varphi_i)$ and $(U_j, {\widetilde U}_j, \Gamma_j, \varphi_j)$, $i, j\in I$,
 are two orbifold charts, with $U_i\,\subset\,U_j$, then there exist a smooth embedding, called {\em change of charts} (or {\em embedding}
 or {\em gluing map})
 $$\rho_{ij}\,\colon\,{\widetilde U}_i\,\longrightarrow\,{\widetilde U}_j$$
(so that ${\widetilde U}_i$ and $\rho_{ij}({\widetilde U}_i)$ are diffeomorphic) such that
$$
  \varphi_i\,=\,\varphi_j\,\circ\,\rho_{ij}.
$$
 \end{enumerate}
 \end{definition}
 
Note that, in most references, the definition given of orbifold chart $(U, {\widetilde U}, \Gamma, \varphi)$ does not explicitly require
 the condition that the finite group $\Gamma$ is such that $\Gamma\subset \GL(n, \mathbb{R})$. But 
 since smooth actions are locally linearizable
(see  \cite[page 308]{Bredon}), any
orbifold has an atlas consisting of {\em linear charts}, that is charts of
the form $(U_i, \RR^n, \Gamma_i, \varphi_i)$ where  $\Gamma_i$ acts on $\RR^n$ via an orthogonal representation
$\Gamma_i\subset \O(n)$. Since $\Gamma_i$ is finite, we can consider an orbifold atlas on a topological space $X$ 
as given in Definition \ref{def:orbifold}.
 
As with smooth manifolds, two orbifold atlases $\SA$ and ${\SA}'$ on $X$ are said to be {\em equivalent} if 
 $\SA\cup{\SA}'$ is also an orbifold atlas.
Equivalent  atlases on $X$ are regarded as defining the same orbifold structure on $X$.
Every orbifold atlas for $X$ is contained in a unique maximal one, and two
orbifold atlases are equivalent if and only if they are contained in the same
maximal one. 

Now, we consider some important points about Definition \ref{def:orbifold}.
Suppose that $X$ is a smooth orbifold, with two orbifold charts $(U_i, {\widetilde U}_i, \Gamma_i, \varphi_i)$ and 
$(U_j, {\widetilde U}_j, \Gamma_j, \varphi_j)$, such that $U_i\,\subset\,U_j$. 
Let $\rho_{ij}\,\colon\,{\widetilde U}_i\,\longrightarrow\,{\widetilde U}_j$ be a change of charts (in the sense of Definition \ref{def:orbifold}). 
Note that $\rho_{ij} \circ \gamma\,\colon\,{\widetilde U}_i\,\longrightarrow\,{\widetilde U}_j$ is also a change of charts, for all $\gamma\in\Gamma_i$.
We will see that, for $\gamma\in\Gamma_i$, there is an element $\widetilde{\gamma} \,\in\Gamma_j$ such that $\rho_{ij} \circ \gamma=\widetilde{\gamma} \circ\rho_{ij}$.
In \cite{MP} it is proved the following result, which was proved by Satake in \cite{S1} under the added assumption that the
fixed point set has codimension at least two.

\begin{proposition}\cite[Proposition A.1]{MP}\label{lemma1}
Let $(U_i, {\widetilde U}_i, \Gamma_i, \varphi_i)$ and 
$(U_j, {\widetilde U}_j, \Gamma_j, \varphi_j)$ be two orbifold charts on $X$, with $U_i\subset U_j$.
If $\rho_{ij},\, \mu_{ij}\,\colon\,{\widetilde U}_i\,\longrightarrow\,{\widetilde U}_j$ are two change of charts,
then there exists a unique $\gamma_j\in\Gamma_j$ such that $\mu_{ij}\,=\,\gamma_j\circ\rho_{ij}$.
\end{proposition}

As a consequence of Proposition \ref{lemma1}, a change of orbifold charts 
$\rho_{ij}\,\colon\,{\widetilde U}_i\,\longrightarrow\,{\widetilde U}_j$
induces an injective homomorphism $f_{ij}\,\colon\,\Gamma_i\,\longrightarrow\,\Gamma_j$ which is
given by
\begin{equation}\label{eq:changecharts}
\rho_{ij}\circ\gamma\,=\,f_{ij}(\gamma)\circ\rho_{ij},
\end{equation}
that is $\rho_{ij}(\gamma\cdot x)\,=\,f_{ij}(\gamma)\cdot \rho_{ij}(x)$,
for all $\gamma\in\Gamma_i$ and $x\in {\widetilde U}_i$.

Also in \cite{MP} it is proved the following.

\begin{lemma}\cite[Lemma A.2]{MP}\label{lemma2}
Let $(U_i, {\widetilde U}_i, \Gamma_i, \varphi_i)$ and 
$(U_j, {\widetilde U}_j, \Gamma_j, \varphi_j)$ be two orbifold charts on $X$, with $U_i\subset U_j$. Consider
$\rho_{ij}\,\colon\,{\widetilde U}_i\,\longrightarrow\,{\widetilde U}_j$ a change of charts  which is equivariant with respect to
the injective homomorphism
$f_{ij}\,\colon\,\Gamma_i\,\longrightarrow\,\Gamma_j$.
If there exists an element $\gamma_j\in\Gamma_j$ such that $\rho_{ij}({\widetilde U}_i)\cap\gamma_j\cdot\rho_{ij}({\widetilde U}_i)\,\not=\emptyset$,
then $\gamma_j\in{\mathrm{Im}}(f_{ij})$, and so $\rho_{ij}({\widetilde U}_i)=\gamma_j\cdot\rho_{ij}({\widetilde U}_i)$.
\end{lemma}

Let $X$ be a smooth orbifold, with an atlas $\{(U_i, {\widetilde U}_i, \Gamma_i, \varphi_i)\}$, and let $p\in X$.
Consider $(U_i, {\widetilde U}_i, \Gamma_i, \varphi_i)$ an orbifold chart around $p$, that is $p= \varphi_{i}(x)\in U_i$ with $x\,\in\,{\widetilde U}_i$, 
and denote by ${\Gamma_i}(x)\,\subset \Gamma_i$ the isotropy
subgroup for the point $x$. Note that, up to conjugation, the group ${\Gamma_i}(x)$ does not depend
on the choice of the orbifold chart around $p$. In fact,  if $(U_i, {\widetilde U}_i, \Gamma_i, \varphi_i)$
is an orbifold chart around $p$ and  
$p= \varphi_{i}(x)=\varphi_{i}(x')\in U_i$ with $x, x'\,\in\,{\widetilde U}_i$, then ${\Gamma_i}(x')$ is conjugate to ${\Gamma_i}(x)$.
(Indeed, there is a group
isomorphism $L_{a}\,\colon\,{\Gamma_i}(x)\,\too\,{\Gamma_i}(x')$ such that, for $\gamma\in{\Gamma_i}(x)$, $L_{a}(\gamma)\,=\,a{\gamma}a^{-1}$ with $a
\,\in \,\GL(n,\RR)$.) 
Moreover, if $(U_j, {\widetilde U}_j, \Gamma_j, \varphi_j)$ is other orbifold chart with
$p=\varphi_j(y)\in U_j$,  
then we have a third orbifold chart $(U_k, {\widetilde U}_k, \Gamma_k, \varphi_k)$ around $p=\varphi_{k}(z)\in U_k$,
together with smooth embeddings 
$$
\rho_{ki}\,\colon\,{\widetilde U}_k\,\longrightarrow\,{\widetilde U}_i, \qquad \rho_{kj}\,\colon\,{\widetilde U}_k\,\longrightarrow\,{\widetilde U}_j, 
$$
and injective homomorphisms
$f_{ki}\,\colon\,\Gamma_k\,\longrightarrow\,\Gamma_i$, $f_{kj}\,\colon\,\Gamma_k\,\longrightarrow\,\Gamma_j$ such that
$\rho_{ki}$ and $\rho_{kj}$ satisfy \eqref{eq:changecharts} with respect to $f_{ki}$ and $f_{kj}$, respectively. Thus, 
$f_{ki}$ and $f_{kj}$ define monomorphisms  $\Gamma_{k}(z)\,\inc\,{\Gamma_i}(x)$
and $\Gamma_{k}(z)\,\inc\,{\Gamma_j}(y)$. But these monomorphisms must be also onto
by Lemma \ref{lemma2}. So, 
$$
\Gamma_{k}(z)\cong \Gamma_{j}(y)\cong \Gamma_i(x).
$$
This justifies that the group ${\Gamma_i}(x)$ is called  the
{\em (local) isotropy group} of $p$, and it is denoted $\Gamma_p$. When $\Gamma_p\not=\Id$, the point $p$ is said to
be a {\em singular} point of the orbifold $X$. 
The points $p$ with $\Gamma_p =\Id$ are called {\em regular} points.
The set of singular points 
$$S=\{p \in X\,\vert\, \Gamma_p\not=\Id\}$$
is called the 
{\em singular locus of the orbifold} $X$ (or {\em orbifold singular set}).
Then $X-S$ is a smooth $n$-dimensional manifold. 

The singular locus can be stratified according to the isotropy groups. For each group $H$,
we have a subset $S_H =\{p\in X\, | \, \Gamma_p=H\}$. It is easily seen that the connected components
of $S_H$ are locally closed smooth submanifolds of $X$. Moreover, the closure $\overline{S}_H$ 
contains components of other $S_{H'}$, with $H<H'$. This is an immediate consequence
of the fact that it holds on every orbifold chart (in an orbifold chart $(U,\widetilde U,\Gamma,\varphi)$, 
the sets $S_H$ are linear subsets of $\widetilde U$). As a consequence, we can give a CW-structure to
$X$ compatible with the stratification, that is, such that the subsets $\overline{S}_H$ are CW-subcomplexes.

Any smooth manifold is a smooth orbifold for which each of the finite groups 
$\Gamma_i$ is the trivial group, so that we get ${\widetilde U}_i$ homeomorphic to $U_i$. 
The most natural examples of orbifolds appear when we take the quotient space  $X=M/\Gamma$
of a smooth manifold $M$ by a f{}inite group $\Gamma$ 
acting smoothly and ef{}fectively on $M$. Let $\pi\colon M\to X$ be  the natural projection.
Note that given un point $p\in M$ with isotropy group $\Gamma_{p}\subset \Gamma$, then
there is a chart $(U, \widetilde U, \phi)$ of $p=\phi(x)\in U$ in $M$, with $U=\phi(\widetilde U)$,
such that $U$ is $\Gamma_{p}$-invariant. 
Then, an orbifold chart of $\pi(p)\in X$ is $(\pi(U), \widetilde U, \Gamma_{p}, \pi\circ\phi)$, 
the change of charts $\rho_{ij}$ are the change of coordinates
on the manifold $M$, and the monomorphisms $f_{ij}$ are the identity map of $\Gamma_{p}$. Such an orbifold 
$$
X=M/\Gamma
$$
 is called {\em effective global quotient} orbifold \cite[Definition 1.8]{Adem}. 

Moreover,  if $M$ is oriented and the action of $\Gamma$ preserves the orientation, then $X$ is an oriented orbifold. 
In general, an orbifold $X$, with atlas $\{(U_i, {\widetilde U}_i, \Gamma_i, \varphi_i)\}$, 
 is {\em oriented} if each $ {\widetilde U}_i$ is oriented, the action of $\Gamma_i$ is 
 orientation-preserving, and all the change of charts 
 $\rho_{ij}\,\colon\,{\widetilde U}_i\,\longrightarrow\,{\widetilde U}_j$ 
 are orientation-preserving.

\begin{definition}[\cite{BG}] \label{def:orbimap}
Let $X$ and $Y$ be two orbifolds (not necessarily of the same dimension) and 
let $\{(U_i, {\widetilde U}_i, \Gamma_i, \varphi_i)\}$ and $\{(V_j, {\widetilde V}_j, \Upsilon_j, \psi_j)\}$ be
atlases for $X$ and $Y$, respectively.
A map $f\,\colon\,X\,\longrightarrow\,Y$ is said to be an {\em orbifold map} (or  {\em smooth map})
if $f$ is a continuous map between the underlying topological spaces, and 
for every point $p\in X$ there are orbifold charts $(U_i, {\widetilde U}_i, \Gamma_i, \varphi_i)$ 
and $(V_i, {\widetilde V}_i, \Upsilon_i, \psi_i)$ for $p$ and $f(p)$, respectively, with $f(U_i)\subset V_i$, a 
differentiable map ${\widetilde f}_i\,\colon\,{\widetilde U}_i\,\longrightarrow\,{\widetilde V}_i$, and a
homomorphism $\varpi_i:\Gamma_i\to \Upsilon_i$ such that ${\widetilde f}_i \circ \gamma=\varpi_i(\gamma)\circ
{\widetilde f}_i$ for all $\gamma\in \Gamma_i$, and 
$$
f_{|\,U_{i}}\circ\varphi_i=\psi_i\circ{\widetilde f}_i.
$$ 
Moreover, $f$ is said to be {\em good} if every map ${\widetilde f}_i$ is 
{\em compatible with the changes of charts} in the following sense:
\begin{enumerate}
\item[i)] if $\rho_{ij}\,\colon\,{\widetilde U}_i\,\longrightarrow\,{\widetilde U}_j$ is a change of charts
for $p$, then there is a change of charts $\mu(\rho_{ij})\,\colon\,{\widetilde V}_i\,\longrightarrow\,{\widetilde V}_j$ for $f(p)$ such that
$$
{\widetilde f}_j\circ\rho_{ij}\,=\,\mu(\rho_{ij})\circ{\widetilde f}_i;  
$$
\item[ii)]  if $\rho_{ki}\,\colon\,{\widetilde U}_k\,\longrightarrow\,{\widetilde U}_i$ is a change of charts
for $p$, then 
$$
\mu(\rho_{ij}\circ \rho_{ki})\,=\, \mu(\rho_{ij})\circ \mu(\rho_{ki}).
$$
\end{enumerate}
Therefore, an orbifold map  $f\,\colon\,X\,\longrightarrow\,Y$ is determined by a smooth map
${\widetilde f}_i\,\colon\,{\widetilde U}_i\,\longrightarrow\,{\widetilde V}_i$, for every orbifold chart   $(U_i, {\widetilde U}_i, \Gamma_i, \varphi_i)$ 
on $X$, such that every  
${\widetilde f}_i$ is $\Gamma_i$-equivariant and compatible with the change of orbifold charts.
\end{definition}
Note that conditions i) and ii) in Definition \ref{def:orbimap} imply that
the composition of orbifold maps is an orbifold map. Moreover, if  $f\,\colon\,X\,\longrightarrow\,Y$ is an orbifold map, then 
there exists an induced homomorphism
from $\Gamma_{p}$ to $\Upsilon_{f(p)}$. Also, considering  $\RR$ as an orbifold, we can define {\em orbifold functions} on an orbifold $X$ as 
orbifold maps $f\,\colon\,X\,\longrightarrow\,\RR$.

Two orbifolds $X$ and $Y$ are said to be {\em diffeomorphic} if there exist orbifold maps
$f\,\colon\,X\,\longrightarrow\,Y$ and $g\,\colon\,Y\,\longrightarrow\,X$ such that
$g\circ f=1_{X}$ and $f\circ g=1_{Y}$, where $1_{X}$ and $1_{Y}$ are the respective identity maps.
Note that a diffeomorphism between orbifolds gives a homeomorphism between the underlying 
topological spaces.

Many of the usual differential geometric concepts that hold for smooth manifolds
also hold for smooth orbifolds; in particular,  the notion of vector bundle \cite[Definition 4.2.7]{BG}.
Using transition maps, orbifold vector bundles can be defined as follows \cite{S2}.

\begin{definition}\label{def:orbibundle}
Let $X$ be a smooth orbifold, of dimension $n$, and let $\{(U_i, {\widetilde U}_i, \Gamma_i, \varphi_i)\}_{i\,\in\,I}$ be
 an atlas on $X$. An {\em orbifold vector bundle over $X$ and fiber $\RR^m$} consists of a smooth orbifold $E$, of dimension $m+n$, and 
an orbifold map $$\pi\,\colon\, E\,\too\, X,$$ 
called {\em projection}, satisfying the following conditions:
\begin{enumerate}
\item[i)] For every orbifold chart $(U_i, {\widetilde U}_i, \Gamma_i, \varphi_i)$ on $X$, there 
exists a homomorphism $$\rho_{i}\,\colon\, \Gamma_{i} \,\too\, \GL(\RR^m)$$ and an orbifold 
chart  $(V_i, {\widetilde V}_i, \Gamma_i, \Psi_i)$ on $E$, such that
$V_i\,=\,\pi^{-1}(U_i)$, ${\widetilde V}_i={\widetilde U}_i \x \RR^m$, the action of $\Gamma_i$
on ${\widetilde U}_i  \x \RR^m$ is the diagonal action (i.e. $\gamma\cdot (x, u)=(\gamma\cdot x, \rho_{i}(\gamma)(u))$, for 
$\gamma\in \Gamma_{i}$, $x\in{\widetilde U}_i$ and for $u\in \RR^m$), and the map
$$
\Psi_i\,\colon\, {\widetilde V}_i={\widetilde U}_i \x \RR^m\,\too\, E_{|U_{i}}\,:=\, \pi^{-1}(U_i)
$$ 
is such that $\pi_{|V_i}\,\circ\,\Psi_i\,=\,\varphi_i\,\circ\,{\mathrm {pr}}_{1}$,
where ${\mathrm {pr}}_{1}\,\colon\, {\widetilde U}_i \x \RR^m\,\too\, {\widetilde U}_i$ is the natural projection, 
$\Psi_i$ is $\Gamma_i$-invariant
for the action of $\Gamma_i$ on ${\widetilde U}_i  \x \RR^m$, and it induces a homeomorphism
$$\Gamma_{i} \backslash ({\widetilde U_i} \x \RR^m)\,\cong\,E_{|U_{i}}$$
\item[ii)] If $(U_i, {\widetilde U}_i, \Gamma_i, \varphi_i)$ and $(U_j, {\widetilde U}_j, \Gamma_j, \varphi_j)$ are
two orbifold charts on $X$, with $U_i\,\subset\,U_j$, and $\rho_{ij}\,\colon\,{\widetilde U}_i\,\longrightarrow\,{\widetilde U}_j$
is a change of charts, then there exists a differentiable map, called {\em transition map}
$$
g_{ij}\,\colon\, {\widetilde U}_i \,\too\, \GL(\RR^m),
$$
and a change of charts
$\lambda_{ij}\,\colon\,{\widetilde V}_i={\widetilde U}_i \x \RR^m\,\longrightarrow\,{\widetilde V}_j={\widetilde U}_j \x \RR^m$ on $E$, such that
$$
\lambda_{ij}(x, u)\,=\,\big(\rho_{ij}(x), g_{ij}(x)(u)\big),
$$
for all $(x, u)\in{\widetilde U}_i \x \RR^m$.
\end{enumerate}

\end{definition}
Note that if $\pi\,\colon\, E\,\too\, X$ is an orbifold vector bundle, and $p\in X$, then the {\em fiber} $\pi^{-1}(p)$
is not always a vector space. In fact, if 
$(U_i, {\widetilde U}_i, \Gamma_i, \varphi_i)$ is an orbifold chart on $X$ around $p=\varphi_i(x)\in X$, then
$$
\pi^{-1}(p)\,\cong\,\Gamma_{p} \backslash (x \x \RR^m)\,\cong\,\Gamma_{p} \backslash \RR^m,
$$
where $\Gamma_{p}=\Gamma_{i}(x)$ is the isotropy group of $p$.
Thus, $\pi^{-1}(p)\,\cong\,\RR^m$ if $p$ is a regular point $(\Gamma_{p}=\Id)$ of $X$,
but $\pi^{-1}(p)$ is not a vector space when $p$ is a singular point. 

\begin{definition}\label{def:orbisection}
A {\em section} (or \emph{orbifold smooth section}) of an orbifold vector bundle $\pi\colon E\,\too\, X$ is an
orbifold map $s \colon X\,\too\, E$ such that $\pi\,\circ\,s\,=\,1_{X}$. Therefore, if $\{(U_i, {\widetilde U}_i, \Gamma_i, \varphi_i)\}$
is an atlas on $X$, then $s$ 
consists of a family of smooth maps $\{s_{i}\,\colon\,{\widetilde U}_i\,\too\, \RR^m\}$, 
such that every  $s_{i}$ is $\Gamma_i$-equivariant and compatible with the changes of charts on $X$
(in the sense of Definition \ref{def:orbimap}). We denote the space of (orbifold smooth) sections
of $E$ by $\SC^\infty(E)$.
\end{definition}

To construct the {\em orbifold tangent bundle} $TX$ of an orbifold $X$, of dimension $n$, we continue to
use the notation of Definition \ref{def:orbibundle}. We define the orbifold charts and the transition maps for $TX$ as follows.
For each orbifold chart $(U_i, {\widetilde U}_i, \Gamma_i, \varphi_i)$
of $X$, we consider the tangent bundle $T{\widetilde U}_i$ over ${\widetilde U}_i$, so $T{\widetilde U}_i\cong{\widetilde U}_i \x \RR^n$.
Take $\rho_{i}\,\colon\, \Gamma_{i} \,\too\, \GL(\RR^n)$ the homomorphism given by the action of $\Gamma_{i}$ on ${\RR}^n$. Then 
$(E_{|\,U_{i}}, {\widetilde U}_{i} \x \RR^n, \Gamma_i, \rho_i, \Psi_{i})$ is an orbifold chart for $TX$, where
$E_{|\,U_{i}} \,=\,\Gamma_{i} \backslash T{\widetilde U}_i$. 
Moreover, if
$\rho_{ij}\,\colon\,{\widetilde U}_i\,\longrightarrow\,{\widetilde U}_j$ is a change of charts for $X$, 
the transition map
$$
g_{ij}\,\colon\, {\widetilde U}_i \,\too\, \GL(\RR^n)
$$ 
for $TX$ is such that
$g_{ij}(x)$ is the Jacobian matrix 
of the map $\rho_{ij}$ at the point $x\in {\widetilde U}_i$. Therefore 
$TX$ is a $2n$-dimensional orbifold, and 
the natural projection $\pi\,\colon\, TX\,\too\, X$ defines a smooth map of orbifolds, with
fibers $\pi^{-1}(p)\,\cong\,\Gamma_{p} \backslash (x \x \RR^n)\,\cong\,\Gamma_{p} \backslash \RR^n$, for $p\in X$.
Therefore, one can consider tangent vectors to $X$ at the point $p\in X$ if $p$ is a regular point.

The orbifold cotangent bundle $T^* X$ and the orbifold tensor bundles are
constructed similarly. Thus, one can consider Riemannian metrics, almost complex structures, 
orbifold forms, connections, etc. 

An (orbifold)  {\em Riemannian metric} $g$ on $X$ 
 is a positive definite symmetric tensor in $T^* X\ox T^*X$. This is equivalent to have, for each orbifold chart  
 $(U_i, {\widetilde U}_i, \Gamma_i, \varphi_i)$ on $X$, a Riemannian metric $g_{i}$ on the open set ${\widetilde U}_i$ that is invariant under the action of $\Gamma_i$ on ${\widetilde U}_i$ ($\Gamma_i$ acts on ${\widetilde U}_i$
by isometries), and the change of charts $\rho_{ij}\,\colon\,{\widetilde U}_i\,\longrightarrow\,{\widetilde U}_j$
are isometries, that is $\rho^{*}_{ij}\big(g_{j}\vert_{\rho_{ij}({\widetilde U}_i)}\big)\,=\,g_{i}$.

An (orbifold)  {\em almost complex structure} $J$ on $X$ 
 is an endomorphism $J\,\colon\,TX\,\longrightarrow\,TX$ such that $J^2\,=\,-\Id$. Thus, $J$ is determined 
 by an almost complex structure $J_i$ on ${\widetilde U}_i$, for every orbifold chart  
 $(U_i, {\widetilde U}_i, \Gamma_i, \varphi_i)$ on $X$, such that the action of $\Gamma_i$ on ${\widetilde U}_i$
 is by biholomorphic maps, and any change of charts  $\rho_{ij}\,\colon\,{\widetilde U}_i\,\longrightarrow\,{\widetilde U}_j$
 is a holomorphic embedding.

An {\em orbifold $p$-form} $\alpha$ on $X$ is a section of $\bigwedge^p T^* X$.
This means that, for each orbifold chart  
 $(U_i, {\widetilde U}_i, \Gamma_i, \varphi_i)$ on $X$, we have
 a differential $p$-form $\alpha_i$ on the open set ${\widetilde U}_i$, such that every 
 $\alpha_i$ is $\Gamma_i$-invariant (i.e. $\gamma^{*}_{i}(\alpha_i)= \alpha_i$, for  $\gamma_i\in\Gamma_i$),
 and any change of charts  $\rho_{ij}\,\colon\,{\widetilde U}_i\,\longrightarrow\,{\widetilde U}_j$
 satisfies  $\rho^{*}_{ij}(\alpha_j)=\alpha_i$.
 
The space of $p$-forms on $X$ is denoted by
$\Omega_{orb}^{p}(X)$. The wedge product of orbifold forms and the 
exterior differential $d$ on $X$ are well defined. Thus, we have 
$$
d\colon\Omega_{orb}^p(X) \,\too\, \Omega_{orb}^{p+1}(X)\, .
$$
The constant sheaf $\RR$ has a resolution
\begin{equation}\label{eb1}
0\,\too\, \RR \,\too\, {\mathbf\varOmega}_{orb}^0 \,\too\,
{\mathbf\varOmega}_{orb}^1 \,\too\, \ldots\, ,
\end{equation}
where ${\mathbf\varOmega}_{orb}^p$
is the sheaf of smooth sections of $\bigwedge^p T^* X$. To prove that
this is a resolution, it is enough to prove that it is exact over any neighborhood of the form
$U\,=\,\widetilde{U}/\Gamma$. As the group $\Gamma$ is finite, it is conjugate to a subgroup of
$\O(n)$,
so we can assume that $\Gamma \subset \O(n)$. We take $\widetilde{U}\,=\,
B_\epsilon(0)$ (the ball in ${\mathbb R}^n$ of radius $\epsilon$ around the origin).
 Then $$0\,\too\,\RR \,\too\, \Omega^0(\widetilde U) \,\too\, \Omega^1(\widetilde U)
\,\too \,\ldots$$ is exact, 
and taking the $\Gamma$-invariant forms, the sequence in \eqref{eb1} is exact as well. (The functor $V\mapsto V^\Gamma$
that sends any vector space $V$ with a $\Gamma$-action, to its $\Gamma$-invariant part, is an exact functor.)
Since \eqref{eb1} is exact, the cohomology of the complex $(\Omega_{orb}^*(X),\,d)$
is isomorphic to the singular cohomology $H^*(X,\, \RR)$. 

We can see more explicitly this isomorphism with duality by pairing with homology classes in singular homology
$H_*(X,\, \RR)$. Recall that we have a CW-complex structure for $X$ such that the singular sets $S_H=\{p\in X \, |
\, \Gamma_p=H\}$ are CW-subcomplexes. Then for an orbifold $k$-form $\alpha$ on $X$ and a $k$-cell
$D\subset X$, we have an integration map $\int_D \alpha$. This is defined as follows: we can assume that
$D$ is inside an orbifold chart $(U,\widetilde U, \Gamma,\varphi)$. Let $D\subset \overline{S}_H$, where $H$ is some
isotropy group, and assume that the interior of $D$ lies in $S_H$. 
Under the quotient map $\pi: \widetilde U\to U$, the preimage of $\pi^{-1}(S_H\cap U)$ is
contained in a linear subspace, and the map $\pi: \pi^{-1}(S_H\cap U) \to S_H\cap U$ is $|H|:1$. 
We define $\int_D \alpha=\frac{|H|}{|\Gamma|}
\int_{\pi^{-1}(D)} \widetilde\alpha$, where $\widetilde\alpha\in \Omega^k(\widetilde U)$ is
the representative of $\alpha$ in the orbifold chart. It is easily seen that this is compatible with the orbifold
changes of charts, and that it satisfies an orbifold version of Stokes theorem.

\begin{remark}\label{rem:forms-co}
Suppose that $X=M/\Gamma$ is an oriented effective global orbifold, that is $X$ is the
quotient of a smooth manifold $M$ by a f{}inite group $\Gamma$ 
acting smoothly and ef{}fectively on $M$. Then, the definition of orbifold forms implies that
any $\Gamma$-invariant differential $k$-form $\alpha$ on $M$ defines an
orbifold $k$-form $\widehat{\alpha}$ on $X$, and vice-versa. Moreover, it is straightforward to
check that the exterior derivative on $M$ preserves $\Gamma$-invariance. Thus, if
$\big(\Omega^{k}(M)\big)^{\Gamma}$ denotes the space of the $\Gamma$-invariant differential $k$-form on $M$,
and ${H^k(M,\, \RR)}^{\Gamma}\subset H^k(M,\, \RR)$ is the subspace of the cohomology classes of degree
$k$ on $M$ such that each of these classes has a representative that is a $\Gamma$-invariant differential $k$-form, then we have 
\begin{equation}\label{orbi-fcoh}
\Omega_{orb}^{k}(X)\,=\,\big(\Omega^{k}(M)\big)^{\Gamma}, \qquad  H^k(X,\, \RR)\,=\,{H^k(M,\, \RR)}^{\Gamma}.
\end{equation}
For any compact supported orbifold $n$-form $\widehat{\alpha}$ on $X$, 
which is by definition a $\Gamma$-invariant compact supported differential $n$-form $\alpha$
on $M$, the integration of $\widehat{\alpha}$ on $X$ is defined by
\begin{equation}\label{orbi-int}
\int_X\, \widehat{\alpha} \,= \, \vert{\Gamma}\vert \int_M\, {\alpha}, 
\end{equation}
where $\vert{\Gamma}\vert$ is the order of the group $\Gamma$. More generally,
one can extend the notion of integration to arbitrary orbifolds by
working in orbifold charts via a partition of unity (\cite[page 34]{Adem},  \cite{S1}).
\end{remark}

\begin{definition}\label{def:orbifolds-minimalmodel}
Let $X$ be an orbifold. A {\it minimal model} for $X$ is a minimal
 model $(\bigwedge V,\,d)$ for the DGA $(\Omega_{orb}^*(X),d)$. The orbifold $X$ is
 {\em formal} if its minimal model is formal (see Section \ref{sec:formality}).
\end{definition}

\begin{proposition} \label{prop:orbifold-model}
  Let $(\bigwedge V,\,d)$ be the minimal model of an orbifold $X$. Then
  $H^*(\bigwedge V)=H^*(X,\, \RR)$, where the latter means singular cohomology
  with real coefficients.
\end{proposition}

For a simply connected orbifold $X$, the dual of the real homotopy vector
space $\pi_i(X)\otimes \RR$ is isomorphic to the space $V^i$ of generators in degree $i$, for any $i$, where
$\pi_i(X)$ is the homotopy group of order $i$ of the underlying topological space in $X$.
In fact, the proof given in \cite{DGMS} for simply connected manifolds, also works 
for simply connected orbifolds (that is, orbifolds for which the topological space $X$ is simply connected).

Moreover, the proof of Theorem \ref{fm2:criterio2} given in \cite{FM}
only uses that the cohomology $H^*(M)$ is a Poincar\'e duality algebra. By \cite{S1}, we know that
the singular cohomology of an orbifold also satisfies a Poincar\'e duality.
Thus, Theorem \ref{fm2:criterio2} also holds 
for compact connected orientable orbifolds. Hence, we have the following lemma.

\begin{lemma}\label{fm2:orbifold} 
Any simply connected compact orbifold of dimension at most $6$
is formal.
\end{lemma}

The notion of formality is also defined for CW-complexes which
have a minimal model $(\bigwedge V,\,d)$. Such a minimal model is
constructed as the minimal model associated to the differential
complex of piecewise-linear polynomial forms \cite{FHT, GM}. 
In particular, we have a minimal model $(\bigwedge V,\,d)$ for orbifolds.

\section{Elliptic differential operators on orbifolds}\label{sec:ellipticoperators}

Here we study elliptic differential operators on complex orbifolds by adapting to these spaces the elliptic operator theory 
on complex manifolds  \cite[Chapter IV]{Wells}.

A {\em complex orbifold}, of complex dimension $n$, is an orbifold $X$ whose orbifold charts are of the form 
$\{(U_i, {\widetilde U}_i, \Gamma_i, \varphi_i)\}$, where ${\widetilde U}_i\,\subset\, \CC^n$,
$\Gamma_i \,\subset\, \GL(n,\CC)$ is a finite group acting on ${\widetilde U}_i$ by biholomorphisms, 
and all the changes of charts $\rho_{ij}\,\colon\,{\widetilde U}_i\,\longrightarrow\,{\widetilde U}_j$ are holomorphic embeddings.
Thus, any complex orbifold has associated an almost complex structure $J$.

If $X$ and $Y$ are complex orbifolds, a map $f\,\colon\,X\,\longrightarrow\,Y$ is said to be an
{\em orbifold holomorphic map} (or  {\em holomorphic map})
if $f$ is a continuous map between the underlying topological spaces, and 
for every point $p\in X$ there are orbifold charts $(U_i, {\widetilde U}_i, \Gamma_i, \varphi_i)$ 
and $(V_i, {\widetilde V}_i, \Upsilon_i, \psi_i)$ for $p$ and $f(p)$, respectively, with $f(U_i)\subset V_i$, and a 
holomorphic map ${\widetilde f}_i\,\colon\,{\widetilde U}_i\,\longrightarrow\,{\widetilde V}_i$ 
such that ${\widetilde f}_i$ is $\Gamma_i$-{\em equivariant} and compatible with changes of charts
(in the sense of Definition \ref{def:orbimap}).

Similarly to orbifold vector bundles, one can define {\em complex orbifold vector bundles}. Let $X$ be a complex orbifold, of complex dimension $n$. 
A {\em complex orbifold vector bundle} over $X$ and fiber $\CC^m$ consists of a complex orbifold $E$, of complex dimension $m+n$, and 
a holomorphic orbifold map $$\pi\,\colon\, E\,\too\, X,$$ 
such that the atlas on $E$ has charts of the type 
$(E_{|U_{i}}, \,{\widetilde U}_i \x \CC^m,  \,\Gamma_i,  \, \rho_i, \, \Psi_i)$,
where $\rho_{i}\colon\Gamma_i \,\too\, \GL(\CC^m)$
is a homomorphism, and 
$$
 \Psi_i\colon{\widetilde U}_i \x \CC^m \,\too\, E_{|U_{i}}\,:=\, \pi^{-1}(U)
$$
is a $\Gamma_i$-invariant map, for the diagonal action of $\Gamma_i$ on ${\widetilde U}_i \x \CC^m$ (the
group $\Gamma_i$ acts on $\CC^m$ via $\rho_i$),
with $\Gamma_i \backslash ({\widetilde U}_i \x \CC^m)\,\cong\,
E_{|U_{i}}$.

A {\em Hermitian metric} $h$ on $X$ is a collection $\{h_i\}$, where each $h_i$ is a Hermitian metric on the open set ${\widetilde U}_i$ of the
(complex) orbifold chart $(U_i,\, {\widetilde U}_i, \,{\Gamma}_i, \,{\varphi}_i)$ on $X$, such that every $h_i$ is 
$\Gamma_i$-invariant, and all the changes of charts $\rho_{ij}\,\colon\,{\widetilde U}_i\,\longrightarrow\,{\widetilde U}_j$
are given by holomorphic and isometric embeddings.  
A slight modification of the usual partition of unity argument shows that every complex orbifold has a Hermitian metric \cite{MM}.

Complex orbifold forms on a complex orbifold and 
the orbifold Dolbeault cohomology will be considered in Section \ref{sec:kahler-orb}. 

Let $E\,\to\, X$ be a complex orbifold vector bundle endowed with a Hermitian metric. 
A {\em Hermitian connection} $\nabla$ on $E$ is defined to be
a collection $\{\nabla_i\}$, where each $\nabla_i$ is a 
$\Gamma_i$-equivariant Hermitian connection on 
${\widetilde U}_i$, for every complex orbifold chart $(U_i,\, {\widetilde U}_i, \,{\Gamma}_i, \,{\varphi}_i)$ on $X$, and such that
$\nabla_i$  is compatible with changes of charts.
Using $\nabla$, we can define
{\em Sobolev norms on sections} of $E$. For a section $s$ supported on a chart $U_i$,
define 
  $$
 ||s||_{W^m(E{|_{U_i}})} \,:=\,\frac{1}{|\Gamma_i|}  ||s||_{W^m({\widetilde U}_i)}\, ,
$$ 
where $W^m$ denotes the usual Sobolev $m$-norm. That is, $||s||_{W^m({\widetilde U}_i)}= \sum_{k=0}^m
|| \nabla_i^k s||_{L^2}$. For orbifold sections $s$ of $E$, we define $||s||_{W^m(E)}= \sum_ i 
|| \rho_i\, s||_{W^m(E{|_{U_i}})}$,
where $\{U_i\}$ is a covering of $X$ by orbifold charts, and $\{\rho_i\}$ a subordinated (orbifold) partition of unity.
The space $W^m(E)$ is the completion with respect to the $W^m$-norm of the space of (orbifold smooth) sections.
In particular, $W^0(E)\,=\,L^2(E)$.
The Sobolev embedding theorem and Rellich's lemma hold for orbifolds 
(the proof in \cite[Chapter IV.1]{Wells} can be extended to orbifolds verbatim). 

A {\em differential operator} $L\in \Diff_k(E,F)$ {\em of order} $k$ between  complex vector bundles $E$
and $F$ is a linear operator which is on an orbifold chart $(U_i,\widetilde U_i,\Gamma_i,\varphi_i)$  of the form 
 \begin{equation}\label{eqn:L=}
 L\,=\,\sum_{|\sigma|\leq k} a_\sigma(x) \frac{D^{|\sigma|}}{D^\sigma x}\, ,
 \end{equation}
where $a_\sigma(x)\in \Hom(E,F)$ is defined on each $\widetilde U_i$ and it is
$\Gamma_i$-equivariant. The {\em symbol} of $L$ is defined as 
 $$
 \sigma_k(L)(x,\xi)\,=\, \sum_{|\sigma|= k} a_\sigma(x) \xi^{\sigma}\, ,
 $$
for $x\,\in \,\widetilde U_i$,  $~\xi \,\in  \, \RR^n$. 
It is easily seen that this defines a symbol $\sigma_k(L)(x,\xi)$, for
$x\in \widetilde U_i$ and $\xi \in T_x^* \widetilde U_i$, which is $\Gamma_i$-equivariant, that
is, an orbifold section of the orbifold bundle $\Hom(E,F) \otimes (T^*X)^{\ox k}$.
We say that $L$ is an {\em elliptic operator} if the symbol of $L$ is
an isomorphism for any $\xi\,\neq\, 0$.

The {\em adjoint} $L^*$ {\em of a differential operator} $L\,\in\, \Diff_k(E,F)$ is the
operator defined by:
 \begin{equation}\label{eqn:L*=}
 \la L(s)\, ,t \ra \,=\, \la s\, , L^*(t) \ra\, ,
 \end{equation} 
for any orbifold sections $s,t$ of $E,F$, respectively. 
It turns out that $L^*\in \Diff_k(F,E)$. For checking this, we go to an orbifold chart
$(U_i,\widetilde U_i,\Gamma_i,\varphi_i)$. Then $L$ is written as (\ref{eqn:L=}). Then the
equality (\ref{eqn:L*=}), for compactly supported $\Gamma_i$-equivariant sections on $\widetilde{U}_i$, shows
that $L^*$ has the form (\ref{eqn:L=}) for suitable coefficients $a_\sigma(x)\in \Hom(F,E)$, also $\Gamma_i$-equivariant.
An operator $L\in \Diff_k(E)\,:=\, \Diff_k(E,E)$ is called {\em self-adjoint} if $L^*\,=\,L$.

\begin{theorem}
Let $L\in \Diff_k(E)$ be self-adjoint and elliptic. Let 
$$\SH_L(E)\,=\,\{ v\in \SC^\infty(E) \,\mid\, L(v)=0\}\, .$$ 
Then there exist linear mappings $H_L, \,G_L\colon\SC^\infty(E) \,\too\, \SC^\infty(E)$
such that
\begin{enumerate}
\item $H_L(\SC^\infty(E))\,=\,\SH_L$ and $\dim \SH_L(E)\,<\,\infty$,
\item $L\circ G_L+ H_L=G_L \circ L + H_L \,=\, \Id$,
\item $H_L,G_L$ extend to bounded operators in $L^2(E)$, and
\item $\SC^\infty(E)\,=\,\SH_L(E) \oplus G_L\circ L(\SC^\infty(E))\,=\,\SH_L(E) \oplus
L\circ G_L(\SC^\infty(E))$, with the decomposition being orthogonal with respect to the $L^2$-metric.
\end{enumerate}
\end{theorem}

\begin{proof}
 The theory in Chapter VI.3 of \cite{Wells} works for orbifolds. A pseudo-differential operator
 is a linear operator $L$ which is locally of the form
$$
u(x)\, \longmapsto\, L(p)u(x)\,=\,\int p(x,\xi) {\widehat u} (\xi) e^{i\la x, \xi\ra} d\xi
$$
for compactly supported $u(x)$, where $p(x,\xi)$ is
a $\Gamma$-invariant function on $T^*\widetilde U\,= \,{\widetilde U}\x \RR^n$
such that the growth conditions in Definition 3.1 of \cite[Chapter VI]{Wells} hold. Note
that $L(p)$ takes $\Gamma$-equivariant sections 
to $\Gamma$-equivariant sections. If we decompose $\SC^\infty(\widetilde U)\,=\,
\SC^\infty(\widetilde U)^\Gamma\oplus D$, 
where $D\,=\,\{ s \,\mid\, \sum_{g\in \Gamma} g^* s\,=\,0\}$, then $L(p)$ maps $D$ to $D$.
 
A pseudo-differential operator $$L\colon\SC^\infty(E)\,\too\,\SC^\infty(E)$$ is of order
$k$ if it extends continuously to
$L\colon W^m(E) \,\too\, W^{m+k}(E)$ for every $m$. Note that locally, $L$ maps
$\Gamma$-equivariant sections
of $W^m(\widetilde U)$ to $\Gamma$-equivariant sections of
$W^{m+k}(\widetilde U)$. In particular, a differential operator of order $m$ is a
pseudo-differential operator of order $m$.
 
First, using the ellipticity of $L$, one constructs a pseudo-differential operator 
$\widetilde L$, such that $L\circ \widetilde L - \Id$ and $\widetilde L\circ L - \Id$ are of order 
$-1$. With this, one can check the regularity of the solutions of the
equation $L v\,=\,0$,
that is $$\SH_L(E)_m\,=\,\{ v \,\in\, W^m(E)\,\mid\, Lv \,=\,0\}
\,\subset\, \SC^\infty(E)\, ,$$ so that 
$\SH_L(E)=\SH_L(E)_m$ for all $m$. Using Rellich's lemma, this proves that
$\SH_L(E)$ is of finite dimension. Now $H_L$ is defined as projection onto $\SH_L(E)$, and $G_L$ is
defined as the inverse
of $L$ on the orthogonal complement to $\SH_L(E)$ and zero on $\SH_L(E)$. With this,
it turns out that $G_L$ is an operator of negative order. The rest of the assertions are now
straightforward.
\end{proof}

Let $E_0,E_1,\ldots, E_N$ be a collection of complex orbifold vector bundles over $X$.
A sequence of differential operators 
 $$
 \SC^\infty(E_0)\,\stackrel{L_0}{\too}\,\SC^\infty(E_1)\,\stackrel{L_1}{\too}\, 
 \SC^\infty(E_2) \,\stackrel{L_2}{\too}\, \ldots \,\stackrel{L_{N-1}}{\too}\, \SC^\infty(E_N)
 $$
 is an {\em elliptic complex} if $L_i\circ L_{i-1}\,=\,0$, $i\,=\,1, \cdots, N-1$, and the
sequence of symbols 
 $$
 0\,\too\, (E_0)_x\,\stackrel{\sigma(L_0)}{\too} \,(E_1)_x \,
\stackrel{\sigma(L_1)}{\too} \,
 (E_2)_x \,\stackrel{\sigma(L_2)}{\too}\, \ldots \,\stackrel{\sigma(L_{N-1})}{\too}\, (E_N)_x \,\too 0
 $$
is exact for all $x\,\in\, X$, $\xi\,\neq\, 0$.
The cohomology groups of the complex are defined to be
 $$
 H^q(E)\,:=\,\frac{\ker L_q}{\im L_{q-1}}\, .
 $$
Writing $E\,=\,\bigoplus_{i=1}^N E_i$, $L\,=\,\sum_{i=1}^{N-1} L_i$, and
$$\Delta\,=\,L^* \, L + L \, L^*$$ with
respect to some fixed Hermitian metric on every $E_i$, $0\, \leq\, i\, \leq\, N$, we
have an elliptic
operator $\Delta\colon\SC^\infty(E)\,\too\,\SC^\infty(E)$. Note that
$\Delta\colon\SC^\infty(E_i)\,\too\,\SC^\infty(E_i)$, for all $i\,=\,0,1,\ldots, N$. 
We denote
 $$
 \SH^j(E)\,=\,\ker (\Delta|_{E_j})\, .
 $$
The following is an analogue of Theorem 5.2 in \cite[Chapter V]{Wells}.

\begin{theorem} \label{thm:2}
 Let $(\SC^\infty(E),L)$ be an elliptic complex equipped with an inner product. Then
the following statements hold:
 \begin{enumerate}
 \item There is an orthogonal decomposition 
$$
\SC^\infty(E)\,=\,\SH(E)\oplus LL^* G(\SC^\infty(E)) \oplus L^*LG (\SC^\infty(E))\, .
$$
\item 
 $\Id=H+\Delta G=H+G\Delta$, $HG=GH=H\Delta=\Delta H=0$, $L\Delta=\Delta L$, $L^*\Delta=\Delta L^*$,
 $LG=GL$, $L^*G=G L^*$, $LH=HL=L^*H=HL^*=0$.
 \item $\dim \SH^j(E)<\infty$, and there is a canonical isomorphism $\SH^j(E) \cong H^j(E)$.
\item $\Delta v\,=\,0 \iff Lv\,=\,L^*v\,=\,0$ for all $v\,\in\, \SC^\infty(E)$.
\end{enumerate}
\end{theorem}

The complex
 $$
\Omega^0_{orb}(X) \stackrel{d}{\too}\Omega^1_{orb}(X)
\stackrel{d}{\too}\Omega^2_{orb}(X) \stackrel{d}{\too} \cdots
\stackrel{d}{\too}\Omega^n_{orb}(X) 
$$
is elliptic. Hence Theorem \ref{thm:2} implies that
\begin{equation}\label{eqn:aa}
H^k(X) \,\cong\, \SH^k(X)\,=\,
\ker (\Delta\colon\Omega^{k}_{orb}(X) \too \Omega^{k}_{orb}(X))\, ,
\end{equation}
where $\Delta\,=\,d d^*+d^* d$.

\section{K\"ahler orbifolds} \label{sec:kahler-orb}

Let $X$ be a complex orbifold, of complex dimension $n$, with an atlas $\{(U_i,\, {\widetilde U}_i, \,{\Gamma}_i, \,{\varphi}_i)\}$. 
As for complex manifolds, we can consider 
{\em orbifold complex forms} on $X$. 
An {\em orbifold complex $k$-form} $\alpha$ on $X$ is given by a complex $k$-form $\alpha_i$ on the open set ${\widetilde U}_i$, for each orbifold chart  
 $(U_i, {\widetilde U}_i, \Gamma_i, \varphi_i)$, and such that every 
 $\alpha_i$ is $\Gamma_i$-invariant and preserved by all
the change of charts. We say that $\alpha$ is {\em bigraded} of type $(p,q)$, with $k=p+q$, if each
$\alpha_i$ is a $(p,q)$-form on ${\widetilde U}_i$.
Denote by  $\Omega^{p,q}_{orb}(X)$ the space of orbifold $(p,q)$-forms on $X$.
Then, we have the type
decomposition of the exterior derivative $d\,=\,\bd+\overline\bd$, where
$$\bd\colon\Omega^{p,q}_{orb}(X) \,\too\, \Omega^{p+1,q}_{orb}(X)\ ~ \text{ and }~ \
\overline\bd\colon\Omega^{p,q}_{orb}(X) \,\too\, \Omega^{p,q+1}_{orb}(X)\, .$$ The (orbifold) Dolbeault
cohomology of $X$ is defined to be
$$
H^{p,q}(X)\,:=\,
\frac{\text{ker}(\overline\bd\colon\Omega^{p,q}_{orb}(X) \too \Omega^{p,
q+1}_{orb}(X))}{\overline\bd (\Omega^{p,q-1}_{orb}(X))}\, .
$$

Fix an orbifold Hermitian metric on $X$. For
any $p\, \geq\, 0$, the complex
 $$
0\too \Omega^{p,0}_{orb}(X) \stackrel{\overline\bd}{\too} \Omega^{p,1}_{orb}(X) \stackrel{\overline\bd}{\too}\Omega^{p,2}_{orb}(X) 
\stackrel{\overline\bd}{\too} \ldots \stackrel{\overline\bd}{\too}\Omega^{p,n}_{orb}(X) \too 0
$$
is elliptic, where $n$ is the complex dimension of $X$. Hence Theorem \ref{thm:2} implies that
 $$
 H^{p,q}(X)\,\cong\, \SH^{p,q}(X)\,=\, \ker (\Delta_{\overline\bd}\colon\Omega^{p,q}_{orb}(X)
\too \Omega^{p,q}_{orb}(X) )\, ,
$$
where $\Delta_{\overline\bd}\,=\,\overline\bd \, {\overline\bd}^*+{\overline\bd}^*
\overline\bd$.

Let $(X,\, J,\, h)$ be a complex Hermitian orbifold, with orbifold complex structure $J$ and Hermitian metric $h$. Thus,
we have an orbifold Riemannian metric $g=\mathrm{Re}\,h$ and an orbifold 2-form $\omega \,\in\, \Omega_{orb}^{1,1}(X)$ 
defined by 
$$\omega=\mathrm{Im}\,h\,,$$
Then, $\omega^n\not=0$, where $n$ is the complex dimension of $X$.

\begin{definition}
A complex Hermitian orbifold $(X, h)$ is called {\em K\"ahler orbifold} if the associated fundamental
form $\omega$ is closed, that is $d\omega\,=\,0$. 
\end{definition}

\begin{proposition} \label{prop:3}
For a compact K\"ahler orbifold, $$\Delta\,=\,2\Delta_{\overline\bd}\, .$$ Therefore $\SH^k(X)
\,=\,\bigoplus_{p+q=k} \SH^{p,q}(X)$.
\end{proposition}

\begin{proof}
This is true on the dense open subset of non-singular points of $X$ by Theorem 4.7 of \cite[Chapter V]{Wells}.
 So it holds everywhere on $X$. 
\end{proof}

\begin{corollary}
For a compact K\"ahler orbifold, $b_k(X)$ is even for $k$ odd.
\end{corollary}

\begin{proof}
Clearly, conjugation gives a map $\Omega^{p,q}_{orb}(X)\,\longrightarrow\, \Omega^{q,p}_{orb}(X)$ that
commutes with $\Delta$ (as this is a real operator). Therefore,  the induced map 
$\SH^{p,q}(X)\,\longrightarrow\, \SH^{q,p}(X)$ is an isomorphism. 
In particular, $h^{p,q}(X)\,=\,h^{q,p}(X)$, where $h^{p,q}(X)=\dim
H^{p,q}(X)$. Thus, $b_k(X)=\sum_{k=p+q} h^{p,q}(X)$ is even for $k$ odd.
\end{proof}

\begin{lemma} \label{lem:dd-lemma}
\mbox{}
\begin{enumerate}
\item Take $\alpha\,\in\, \Omega^{p,q}_{orb}(X)$ with $\bd \alpha\,=\,0$.
If $\alpha\,=\,{\overline\bd} \beta$ for some $\beta$, then
there exists $\psi$ such that $\alpha\,=\,\bd{\overline\bd} \psi$.

\item Take $\alpha\,\in\, \Omega^{p,q}_{orb}(X)$ with ${\overline\bd} \alpha\,=\,0$.
If $\alpha\,=\,\bd \beta$ for some $\beta$, then
there exists $\psi$ such that $\alpha\,=\,\bd{\overline\bd} \psi$.
\end{enumerate}
\end{lemma}

\begin{proof}
 Using Theorem \ref{thm:2},
 $$
 \alpha\,=\, H\alpha + \Delta_{\overline\bd} G \alpha
\,=\,H\alpha+ {\overline\bd}\, {\overline\bd}^*G \alpha+{\overline\bd}^*
{\overline\bd} G\alpha\, ,
 $$
where $G\,=\,G_{\overline\bd}$ is the Green's operator associated to $\overline\bd$.
As $\alpha\,=\,{\overline\bd} \beta$, the cohomology class represented by
$\alpha$ vanishes, so $H\alpha\,=\,0$. Then, since $G$ commutes
with $\overline\bd$, we have ${\overline\bd} G\alpha\,=\, G{\overline\bd} \alpha
\,=\,0$. Hence $\alpha\,=\, {\overline\bd} \,{\overline\bd}^*G \alpha\,=\,
{\overline\bd} G ({\overline\bd}^*\alpha)$.

Now ${\overline\bd}^*\,=\,\sqrt{-1} [\Lambda,\bd]$, where $\Lambda\,=\,L_\omega^*$ and
$L_\omega(\beta)\,=\,\omega\wedge \beta$.
So ${\overline\bd}^*\alpha\,=\,-\sqrt{-1}\bd \Lambda\alpha$, because $\bd\alpha\,=\,0$. Hence
$\alpha\,=\, {\overline\bd} G (-\sqrt{-1} \bd \Lambda\alpha)\,=\,-
\sqrt{-1} \, {\overline\bd} \bd (G\Lambda\alpha)$. Therefore, taking
$\psi\,=\,-\sqrt{-1} G\Lambda\alpha$, we conclude the proof of the first part.

The proof of the second part is identical.
\end{proof}

\begin{theorem}\label{thm:K-orb-formal}
 Let $X$ be a compact K\"ahler orbifold. Then $X$ is formal.
\end{theorem}

\begin{proof}
We have to show that $(\Omega^*_{orb}(X),\,d)$ and 
$(H^*(X),0)$ are quasi-isomorphic differential graded commutative
algebras (DGA).
 
Consider the DGA $(\ker \bd,\overline\bd)$. We will show that
 $$
\imath\colon(\ker \bd,\overline\bd) \,\hookrightarrow\, (\Omega^*_{orb}(X),\,d)
 $$
is a quasi-isomorphism. To prove surjectivity, we can take 
a $(p,q)$-form $\alpha$ which is $d$-closed (see Proposition \ref{prop:3}). If
$d\alpha\,=\,0$, then $\bd \alpha\,=\,0$ and ${\overline\bd} \alpha\,=\,0$. So
$\alpha\,\in\, \ker \bd$
and $\imath^* [\alpha]\,=\,[\alpha]$. For injectivity, take $\alpha\,\in\,
\ker\bd$ such that $\imath^*[\alpha]\,=\,0$. Then ${\overline\bd} \alpha\,=\,0$ and 
$\alpha\,=\,d\beta$, for some form $\beta$. Therefore, $\alpha\,=\,
\bd\beta+{\overline\bd}\beta$. Thus we have
${\overline\bd}(\bd\beta)\,=\,0$. By Lemma \ref{lem:dd-lemma}, we have
that $\bd\beta\,=\,\bd{\overline\bd} \psi$ for some $\psi$. Hence
$\alpha\,=\,{\overline\bd}\beta+\bd{\overline\bd}\psi\,=\,{\overline\bd} (\beta-\bd \psi-{\overline\bd} \psi)$.
Note that $\bd (\beta-\bd\psi-{\overline\bd}\psi)\,=\,\bd \beta-\bd
{\overline\bd} \psi\,=\,0$, so 
$\beta-\bd\psi-{\overline\bd}\psi \,\in\, \ker \bd$.

Next we will show that the projection given by 
 $$
 H\colon(\ker \bd,\overline\bd) \,\too\, (\SH^*_{\overline\bd}(X),0)
 $$
is a quasi-isomorphism.

Let $\alpha\,\in\, \ker\bd \cap \ker \overline\bd$. Then
${\overline\bd}^* \alpha\,=\, \sqrt{-1}[\Lambda,\bd]\alpha\,=\,- \sqrt{-1} \bd
(\Lambda \alpha)$. So 
 $$
 \alpha\,=\,H\alpha+ G({\overline\bd}\, {\overline\bd}^*\alpha+
{\overline\bd}^*{\overline\bd}\alpha)
\,=\, H\alpha - \sqrt{-1} G{\overline\bd} \bd(\Lambda\alpha)\, ,
 $$
that is $\alpha\,=\,H\alpha+ \bd{\overline\bd}\psi$, for some $\psi$. Therefore, if $H\alpha\,=\,0$, then
$\alpha\,=\,{\overline\bd} (\bd \psi)$, with $\bd \psi \,\in\, \ker \bd$. This proves
injectivity.

Now suppose $\alpha\,=\,H\alpha+ \bd{\overline\bd}\psi$ and
$\beta\,=\,H\beta+\bd{\overline\bd} \phi$.
So $$\alpha\wedge \beta\,=\,H\alpha\wedge H\beta+ \bd{\overline\bd} \Phi$$ for some
$\Phi$, hence $H(\alpha\wedge \beta)\,=\,H\alpha\wedge H\beta$. This implies
that $H$ is a DGA map.

Finally, let us show surjectivity of $H$. Take $\alpha$ to be harmonic. Then ${\overline\bd} \alpha\,=\,0$
and ${\overline\bd}^* \alpha\,=\,0$. Since $\Delta\,=\,2\Delta_{\overline\bd}$, we also have
$d\alpha\,=\,0$ and $\bd \alpha\,=\,0$. So $H([\alpha])\,=\,\alpha$.
\end{proof}

The hard Lefschetz property is proved in \cite{Wang-Zaffran}, but we shall give a proof 
with the current techniques for completeness.

\begin{theorem} \label{thm:hard-Lefschetz}
Let $(X,\omega)$ be a compact K\"ahler orbifold. 
 Then the map
\begin{equation}\label{eqn:oo}
L_\omega^{n-k}\colon H^k(X) \,\too\, H^{2n-k}(X)
\end{equation}
is an isomorphism for $0\leq k \leq n$.
\end{theorem}

\begin{proof}
It is enough to see that (\ref{eqn:oo}) is onto, since by Poincar\'e duality
both spaces have the same dimension. 
As $[L_\omega,\Delta]=0$, then $L_\omega$ sends harmonic forms to harmonic forms. Therefore
we have to see that 
 $$
L_\omega^{n-k}\colon\SH^k(X) \too \SH^{2n-k}(X)
 $$ 
is surjective. We shall prove this by induction on $k=0,1,\ldots, n$.
Take a harmonic $(2n-k)$-form $a$. By induction on $k$ applied to $L_\omega(a)$,
we have that $L_\omega(a)=L_\omega^{n-k+2}(c)$ for a $(k-2)$-form $c$. Therefore
$a'=a-L_\omega^{n-k+1}(c)$ is \emph{primitive}, $L_\omega(a')=0$. Let us see that
the Lefschetz map is surjective for a primitive $a'$. 

We have that $[\Lambda,L_\omega]\,=\,n-p$ on $p$-forms. As $\Lambda a\,=\,0$, we have 
$L_\omega(\Lambda a')\,=\,(k-n) a'$, so $(L_\omega(\Lambda a'))^{n-k}\,=\,c\, a'$, for a constant $c$.
Using repeatedly that $L_\omega\Lambda \,=\,\Lambda L_\omega + c\Id$, for (different constants $c$'s),
we get that $L_\omega^{n-k}\Lambda^{n-k} a'\,=\,c\, a'$, for another constant $c$.
The map
 $$
L_\omega^{n-k}\colon\Omega^k_{orb}(X) \too \Omega^{2n-k}_{orb}(X)
 $$
is an isomorphism (it is a bundle isomorphism). So the above constant $c$ is nonzero. Therefore,
$a'\,=\,L_\omega^{n-k}(b)$ with $b\,=\,\frac1c\Lambda^{n-k}(a')$ a harmonic $k$-form
(since $\Lambda$ also sends harmonic forms to harmonic forms).
This finishes the proof of the theorem.
\end{proof}

\section{Symplectic orbifolds with no K\"ahler orbifold structure}\label{so:nonko}

We shall include two examples of symplectic orbifolds, of dimensions $6$ and $8$, 
taken from the constructions in \cite{BaFeMu} and \cite{FM1}, which cannot admit the structure of an
orbifold K\"ahler manifold. The first one because it does not satisfy the hard Lefschetz property, and the second
one because it is non-formal. Both admit complex and symplectic (orbifold) structures.

Before going to those examples, let us recall the definition of a symplectic orbifold.
\begin{definition}
A symplectic orbifold $(X, \omega)$ consists of a $2n$-dimensional orbifold $X$ and an orbifold $2$-form $\omega$ such
that $d\omega=0$ and $\omega^n>0$ everywhere.
\end{definition}

Note that if $(M, \Omega)$ is a symplectic manifold, with symplectic form $\Omega$,
 and $\Gamma$ is a f{}inite group acting ef{}fectively
on $M$ and preserving $\Omega$, then $X=M/\Gamma$ is a symplectic orbifold. In fact, by 
Remark \ref{rem:forms-co}, $X=M/\Gamma$ is an orbifold, and
the symplectic form $\Omega$ descends to $X$ via the natural projection 
$\pi: M\to X$. The map $\pi$ is differentiable in the orbifold sense (actually it is a submersion).

\subsection{$6$-dimensional example}

Consider the complex Heisenberg group $H_\CC$, that is the
complex nilpotent Lie group of (complex) dimension 3 consisting of matrices of the form
\[
 \begin{pmatrix} 1&u_2&u_3\\ 0&1&u_1\\ 0&0&1\end{pmatrix}.
\]
In terms of the natural
(complex) coordinate functions $(u_1,u_2,u_3)$ on $H_\CC$, we have
that the complex $1$-forms $\mu\,=\,du_1$, $\nu\,=\,du_2$ and
$\theta\,=\,du_3-u_2 \, du_1$ are left invariant, and
\[
 d\mu\,=\,d\nu\,=\,0, \quad d\theta\,=\,\mu\wedge\nu\, .
\]
Let $\Lambda \subset \CC$ be the lattice generated by $1$
and $\zeta\,=\,e^{2\pi i/6}$, and consider the discrete subgroup
$\Gamma\subset H_\CC$ formed by the matrices in which $u_1,u_2,u_3 \in \Lambda$. We def{}ine the compact (parallelizable)
nilmanifold
\[
 M\,=\,\Gamma \backslash H_\CC\, .
\]
We can describe $M$ as a principal torus bundle
\[
 T^2\,=\,\CC/\Lambda \,\too\, M \,\too\, T^4\,=\,(\CC/\Lambda)^2
\]
by the projection $(u_1,\,u_2,\,u_3) \,\longmapsto\, (u_1,\,u_2)$.

Consider the action of the f{}inite group $\ZZ_6$ on $H_\CC$ given by the generator
\begin{eqnarray*}
 \rho\colon H_\CC & \longrightarrow & H_\CC\\
 (u_1, u_2,u_3) &\longmapsto & (\zeta^4\, u_1, \zeta\, u_2,\zeta^5\, u_3) .
\end{eqnarray*}
For this action, clearly $\rho(p\cdot q)\,=\,\rho(p)\cdot \rho(q)$,
for all $p, q \,\in\, H_\CC$, where $\cdot$ denotes the natural group
structure of $H_\CC$. Moreover, we have $\rho(\Gamma)\,=\,\Gamma$. Thus, $\rho$
induces an action on the quotient $M\,=\,\Gamma\backslash H_\CC$. Let $\rho\colon M
\,\longrightarrow\, M$ be the $\ZZ_6$-action. The action on 1-forms is given by
\[
 \rho^*\mu= \zeta^4\, \mu, \quad
 \rho^*\nu= \zeta\, \nu,\quad
 \rho^*\theta= \zeta^5\,\theta.
\]

\begin{proposition}\label{ex:6-dim}
$X\,=\,M/\ZZ_{6}$ is a simply connected, compact, formal  $6$-orbifold
 admitting complex and symplectic structures.
\end{proposition}

\begin{proof}
Since the $\ZZ_6$-action on $M$ is effective, the quotient space $X\,=\,M/\ZZ_{6}$ is an orbifold.
(The singular points of $X$ are determined in \cite[Section 4]{BaFeMu}.)
Clearly $X$ is compact since $M$ is compact. 
In \cite[Proposition 6.1]{BaFeMu}, it is proved that the 6-orbifold $X$ (denoted by 
$\widehat{M}$ in \cite{BaFeMu}) is simply connected. Then, $X$ is formal because any simply connected 
compact orbifold of dimension 6 is formal by Lemma \ref{fm2:orbifold}.

The orbifold $X$ has a complex orbifold structure, as in Proposition \ref{prop:1}. 
We def{}ine the complex $2$-form $\omega$ on $M$ by
 \begin{equation}\label{eqn:omega}
 \omega\,=\,- \sqrt{-1} \,\mu \wedge \bar\mu + \nu\wedge \theta + \bar\nu \wedge \bar\theta\, .
 \end{equation}
Clearly, $\omega$ is a real closed $2$-form on $M$ such that $\omega^3 \,>\, 0$,
so, $\omega$ is a symplectic form on $M$. Moreover, the form $\omega$ is 
$\ZZ_{6}$-invariant. Indeed, 
$ \rho^*\omega=- i \,\mu \wedge \bar\mu + \zeta^6 \nu\wedge \theta +\zeta^{-6}\bar\nu \wedge \bar\theta=\omega$.
Therefore $X$ is a {symplectic $6$-orbifold}, with the symplectic form
$\widehat\omega$ induced by $\omega$. 
\end{proof}

In order to prove that $X$ does not admit any K\"ahler structure, we are going to check that it does
not satisfy the 
hard Lefschetz property for any symplectic form. We compute the cohomology of $X$. 
By a theorem of Nomizu theorem \cite{No}, the cohomology of the nilmanifold $M$ is:
\begin{align*}
 H^0(M,\,\CC) = \ & \la 1\ra,\\
 H^1(M,\,\CC) =\ & \la [\mu], [\bar\mu],[\nu],[\bar\nu]\ra,\\
 H^2(M,\,\CC) = \ & \la [\mu \wedge \bar\mu], [\mu \wedge\bar \nu], [\bar\mu \wedge \nu], [\nu\wedge \bar\nu],
[\mu\wedge \theta], [ \bar\mu\wedge \bar\theta], [\nu\wedge\theta],[ \bar\nu\wedge \bar\theta]\ra,\\
 H^3(M,\,\CC) = \ & \la [\mu\wedge\bar\mu\wedge \theta], [\mu\wedge\bar\mu\wedge \bar\theta],
 [\nu\wedge\bar\nu\wedge \theta], [\nu\wedge\bar\nu\wedge \bar\theta],
 [\mu\wedge\nu\wedge \theta], [\bar\mu \wedge \bar\nu \wedge \bar\theta]\\
 & [\mu \wedge \bar\nu \wedge \theta], [\mu \wedge \bar\nu \wedge \bar\theta],
 [\bar\mu \wedge \nu \wedge \theta],[\bar\mu \wedge \nu \wedge \bar\theta]\rangle,\\
H^4(M,\,\CC) = \ & \la [\mu \wedge \bar\mu \wedge \nu \wedge \theta],
 [\mu \wedge \bar\mu \wedge \bar\nu \wedge \bar\theta],
 [\bar\mu \wedge \nu \wedge \bar\nu \wedge \bar\theta],
 [\mu \wedge \nu \wedge \bar\nu \wedge \theta],\\
 & [\mu \wedge \bar\mu \wedge \theta \wedge \bar\theta],
 [\nu \wedge \bar\nu \wedge \theta \wedge \bar\theta],
 [\mu \wedge \bar\nu \wedge \theta \wedge \bar\theta],
 [\bar\mu \wedge \nu \wedge \theta \wedge \bar\theta]\ra , \\
 H^5(M,\,\CC) = \ & \la [\mu \wedge \bar\mu \wedge \nu \wedge
 \theta \wedge \bar\theta],
 [\mu \wedge \bar\mu \wedge \bar\nu \wedge
 \theta \wedge \bar\theta],
 [\mu \wedge \nu \wedge \bar\nu \wedge
 \theta \wedge \bar\theta], [\bar\mu \wedge \nu \wedge \bar\nu \wedge
 \theta \wedge \bar\theta] \ra, \\
 H^6(M,\,\CC) =\ & \la [\mu \wedge \bar\mu \wedge \nu \wedge \bar\nu
 \wedge \theta \wedge \bar\theta]\ra.
\end{align*}
According with \eqref{orbi-fcoh}, any $\ZZ_6$-invariant $k$-form on $M$ defines an
orbifold $k$-form on $X$, and vice-versa. Moreover, the cohomology $H^*(X)\,=\,H^*(M)^{\ZZ_6}$ is: 
\begin{align*}
H^0(X,\,\CC) = \ & \la 1\ra,\\
H^1(X,\,\CC) = \ &0,\\
H^2(X,\,\CC) = \ & \la [\mu \wedge \bar\mu], [\nu\wedge \bar\nu],
 [\nu\wedge\theta],[\bar\nu\wedge \bar\theta]\ra,\\
H^3(X,\,\CC) = \ & 0,\\
H^4(X,\,\CC) = \ & \la [\mu \wedge \bar\mu \wedge \nu \wedge \theta],
 [\mu \wedge \bar\mu \wedge \bar\nu \wedge \bar\theta],
 [\mu \wedge \bar\mu \wedge \theta \wedge \bar\theta],
 [\nu \wedge \bar\nu \wedge \theta \wedge \bar\theta]\ra ,\\
H^5(X,\,\CC) = \ & 0 , \\
H^6(X,\,\CC) = \ & \la [\mu \wedge \bar\mu \wedge \nu \wedge \bar\nu
 \wedge \theta \wedge \bar\theta]\ra,
\end{align*}
where we use the same notation for the $\ZZ_6$-invariant forms on $M$ and those
induced on the orbifold $X$.

The cohomology class
\[
[\beta]\,=\,[\nu\wedge\bar\nu] \,\in\, H^2(X)
\] 
satisfies the equation $[\beta]\wedge [\alpha_1]\wedge [\alpha_2]\,=\,0$ for any 
$[\alpha_1],\,[\alpha_2]\,\in\, H^2(X)$. Therefore this class is always in the kernel of 
 $$
 L_{\omega'}\colon H^2(X) \,\too\, H^4(X),
 $$
for any (orbifold) symplectic form $\omega'$. So we have the following:

\begin{proposition}\label{non-Lefschetz}
The orbifold $X$ does not admit an orbifold K\"ahler structure since it does not satisfy
the hard Lefschetz property for any symplectic form.
\end{proposition}

\subsection{$8$-dimensional example}

Consider again the complex Heisenberg group $H_{\CC}$ and set $G=H_\CC \times \CC$, where $\CC$ is the additive group of
complex numbers. We denote by $u_4$ the coordinate function
corresponding to this extra factor. In terms of the natural
(complex) coordinate functions $(u_1,u_2,u_3,u_4)$ on $G$, the complex $1$-forms $\mu=du_1$, $\nu=du_2$,
$\theta=du_3-u_2 \, du_1$ and $\eta=du_4$ are left invariant, and
$$
 d\mu\,=\,d\nu\,=\,d\eta\,=\,0, \quad d\theta\,=\,\mu\wedge\nu\, .
$$

Let $\Lambda \subset \CC$ be the lattice generated by $1$
and $\zeta=e^{2\pi \sqrt{-1}/3}$,
and consider the discrete subgroup
$\Gamma\subset G$ formed by the matrices in which $u_1,u_2,u_3,
u_4 \in \Lambda$. We define the compact (parallelizable)
nilmanifold
 $$
 M\,=\,\Gamma \backslash G\, .
 $$
We can describe $M$ as a principal torus bundle
 $$
 T^2\,=\,\CC/\Lambda \,\too \,M \,\too \,T^6\,=\,(\CC/\Lambda)^3\, ,
 $$
by the projection $(u_1,\,u_2,\,u_3,\,u_4) \,\longmapsto\, (u_1,\,u_2,\,u_4)$.

Now introduce the following action of the finite group $\ZZ_{3}$
 \begin{eqnarray*}
 \rho\colon G & \longrightarrow & G\\
 (u_1, u_2,u_3, u_4) &\longmapsto & (\zeta\, u_1, \zeta\, u_2,\zeta^2\, u_3,\zeta\, u_4) .
 \end{eqnarray*}
Note that $\rho(p\cdot q)\,=\,\rho(p)\cdot \rho(q)$,
for $p, q \in G$, where the dot denotes the natural group
structure of $G$. The map $\rho$ is a particular case of a
homothetic transformation (by $\zeta$ in this case) which is well
defined for all nilpotent simply connected Lie groups with graded
Lie algebra. Moreover $\rho(\Gamma)=\Gamma$, therefore $\rho$
induces an action on the quotient $M=\Gamma\backslash G$. 
 This action is free away from $3^4$ fixed points corresponding to $u_i =
n/(1-\zeta)$, for $n = 0,1$ and $2$. 

The
action on the forms is given by
 $$
 \rho^*\mu= \zeta\, \mu, \quad
 \rho^*\nu= \zeta\, \nu,\quad
 \rho^*\theta= \zeta^2\,\theta, \quad
 \rho^*\eta= \zeta\,\eta.
 $$

\begin{proposition} \label{prop:1}
$X=M/\ZZ_{3}$ is an {$8$-orbifold} admitting complex and symplectic structures.
\end{proposition}

\begin{proof}
Just as in Proposition \ref{ex:6-dim}, it turns out that $X$ is an $8$-orbifold since the $\ZZ_{3}$-action on $M$ is 
effective. The nilmanifold $M$ is a complex manifold whose complex structure $J$
coincides with
the multiplication by $\sqrt{-1}$ on each tangent space $T_{p}M$, $p\in M$. 
Then one can check that $J$ commutes with the
$\ZZ_3$-action $\rho$ on $M$, that is $(\rho_{*})_p\circ J_{p} = J_{\rho(p)}\circ (\rho_{*})_p$,
for any point $p\in M$. Hence, $J$ induces a complex structure on the quotient $X=M/\ZZ_{3}$.

The complex $2$-form
 $$
 \omega\,=\, \sqrt{-1} \,\mu \wedge \bar\mu + \nu\wedge \theta +
 \bar\nu \wedge \bar\theta+ i \,\eta \wedge \bar\eta
 $$
is actually a real form which is clearly closed and
which has the property that $\omega^4\,\not=\,0$. Thus $\omega$ is a symplectic
form on $M$. Moreover, $\omega$ is $\ZZ_3$-invariant. Hence the
space $X\,=\,M/\ZZ_{3}$ is a symplectic orbifold, with the symplectic form
$\widehat\omega$ induced by $\omega$. 
\end{proof}

The orbifold $X$ does not admit a K\"ahler orbifold structure because it is non-formal, as shown in the following theorem.

\begin{theorem}\label{thm:nonformal}
 The orbifold $X$ is non-formal.
\end{theorem}

\begin{proof}
We start by considering the nilmanifold $M$. Consider the following closed forms:
 $$
 \alpha = \mu \wedge\bar\mu, \quad
 \beta_1 = \nu\wedge\bar\nu,\quad
 \beta_2 = \nu\wedge\bar\eta,\quad
 \beta_3 = \bar\nu\wedge\eta .
 $$
Then
 $$
 \alpha \wedge \beta_1 = d(-\theta\wedge\bar\mu\wedge\bar\nu),
\quad
 \alpha \wedge \beta_2 = d(-\theta\wedge\bar\mu\wedge\bar\eta),
\quad
 \alpha \wedge \beta_3 = d(\bar\theta\wedge \mu\wedge\eta).
 $$
All the forms $\alpha$, $\beta_1$, $\beta_2$, $\beta_3$,
$\xi_1=-\theta\wedge\bar\mu\wedge\bar\nu$,
$\xi_2=-\theta\wedge\bar\mu\wedge\bar\eta$ and $\xi_3=\bar\theta\wedge
\mu\wedge\eta$ are $\ZZ_3$-invariant. Hence by \eqref{orbi-fcoh} they descend to orbifold forms (denoted with a 
${}^\sim$) on the quotient $X=M/\ZZ_3$. 

We consider the $a$-Massey product 
 $$
 \la a; b_1,b_2,b_3\ra,
 $$
for $a=[\tilde{\alpha}], b_i=[\tilde{\beta_i}]\in H^2(X)$, $i=1,2,3$. By Nomizu's theorem
mentioned earlier, the cohomology of $M$ up to degree 3 is
\begin{align*}
 H^0(M,\CC) = \ & \la 1\ra,\\
 H^1(M,\CC) =\ & \la [\mu], [\bar\mu],[\nu],[\bar\nu],[\eta],[\bar\eta]\ra,\\
 H^2(M,\CC) = \ & \la [\mu \wedge \bar\mu], [\mu \wedge\bar \nu], [\mu \wedge \theta],[\mu \wedge \eta],[\mu \wedge \bar\eta],[\bar\mu \wedge \nu], [\bar\mu \wedge \bar\theta],[\bar\mu \wedge \eta],[\bar\mu \wedge \bar\eta],[\nu\wedge \bar\nu],\\
 & [\nu\wedge \theta],[\nu\wedge \eta],[\nu\wedge \bar\eta],[\bar\nu \wedge \bar\theta],[\bar\nu \wedge \eta],[\bar\nu \wedge \bar\eta],[\bar\eta \wedge \bar\eta]\ra,\\
 H^3(M,\CC) = \ & A\oplus\bar{A}\, 
\end{align*}

where
\begin{align*}
A = & \la [\mu\wedge\bar\mu\wedge\bar\theta],[\mu\wedge\bar\mu\wedge \eta],[\mu\wedge\nu\wedge \theta],[\mu\wedge\bar\nu\wedge \bar\theta],[\mu\wedge\bar\nu\wedge \eta],[\mu\wedge\theta\wedge\eta],[\mu\wedge\eta\wedge\bar\eta],\\
&[\bar\mu\wedge\nu\wedge\bar\theta],[\bar\mu\wedge\nu\wedge\eta],[\bar\mu\wedge\bar\theta\wedge\eta],[\nu\wedge\bar\nu\wedge\bar\theta],[\nu\wedge\bar\nu\wedge\eta],[\nu\wedge\theta\wedge\eta],[\nu\wedge\eta\wedge\bar\eta],\\
 &[\bar\nu\wedge\bar\theta\wedge\eta]\rangle.
\end{align*}

Now $\ZZ_3$ acts on $A$ by multiplication with $\zeta$ and on $\bar A$ by multiplication with $\bar\zeta$, hence $H^3(X)=H^3(M)^{\ZZ_3}=0$.
By \cite[Proposition 2.7]{CFM}, the $a$-Massey product $\la a; b_1,b_2,b_3\ra$ has no indeterminacy.

We denote by $q$ the projection $M\,\longrightarrow\, X$, and compute
\begin{align*}
\la a; b_1,b_2,b_3\ra & \ = [\tilde\xi_1\wedge\tilde\xi_2\wedge\tilde\beta_3
+\tilde\xi_2\wedge\tilde\xi_3\wedge\tilde\beta_1
+\tilde\xi_3\wedge\tilde\xi_1\wedge\tilde\beta_2]=\\
&\ = q_*[\xi_1\wedge\xi_2\wedge\beta_3
+\xi_2\wedge\xi_3\wedge\beta_1
+\xi_3\wedge\xi_1\wedge\beta_2]=\\
& \ = 2q_*[\theta\wedge \mu\wedge\nu\wedge\eta\wedge \bar\theta\wedge
\bar\mu\wedge\bar\nu\wedge\bar\eta]
\end{align*}

which is non-zero,  since by \eqref{orbi-int} we have
\begin{align*}
\int_X\la a; b_1,b_2,b_3\ra & \ = 2\int_Xq_*[\theta\wedge \mu\wedge\nu\wedge\eta\wedge \bar\theta\wedge \bar\mu\wedge\bar\nu\wedge\bar\eta]=\\
&\ = 6\int_M[\theta\wedge \mu\wedge\nu\wedge\eta\wedge \bar\theta\wedge \bar\mu\wedge\bar\nu\wedge\bar\eta]\neq 0.
\end{align*}

By Theorem \ref{theo:amassey and formality} and Definition \ref{def:orbifolds-minimalmodel}, the orbifold $X$ is non-formal.
\end{proof}

\section{Simply connected Sasakian manifolds}\label{non-formal-sasakian}

First, we recall some definitions and results on Sasakian manifolds (see
\cite{BG} for more details).

An odd-dimensional Riemannian manifold $(N,g)$ is {\em Sasakian} if its 
cone $(N\times{\mathbb{R}}^+, g^c = t^2 g+dt^2)$ is K\"ahler, that
is the cone metric $g^c \,=\, t^2 g+dt^2$ admits a compatible
integrable almost complex structure $J$ so that
$(N\times{\mathbb{R}}^+, g^c = t^2 g+dt^2, J)$ is a K\"ahler
manifold. In this case the Reeb vector field $\xi\,=\,J\partial_t$ is
a Killing vector field of unit length. The corresponding $1$-form
$\eta$ defined by $\eta(X)\,=\,g(\xi,\,X)$, for any vector field $X$ on
$N$, is a contact form, meaning
$\eta\wedge ({d} \eta)^n\,\not=\,0$ at every point of $N$, where $\dim N\,=\,2n+1$.

A Sasakian structure on $N$ is called \textit{quasi-regular} if there is a positive
integer $\delta$ satisfying the condition
that each point of $N$ has a coordinate chart
$(U\, ,t)$ with respect to $\xi$ (the coordinate $t$ is in the direction of $\xi$)
such that each leaf of $\xi$ passes through $U$ at most $\delta$ times. If
$\delta\,=\, 1$, then the Sasakian structure is called \textit{regular}. (See
\cite[p. 188]{BG}.) A result of \cite{OV} says that if $N$ admits a Sasakian structure,
then it admits also a quasi-regular Sasakian structure.

If $M$ is a K\"ahler manifold whose K\"ahler form $\omega$ 
defines an integral cohomology class, then the total space of the circle bundle 
$S^1 \hookrightarrow N \stackrel{\pi}{\longrightarrow} M$ with Euler class $[\omega]\in 
H^2(M,\mathbb{Z})$ is a regular Sasakian manifold with contact form $\eta$ such that $d 
\eta \,=\, \pi^*(\omega)$. The converse also holds: if $N$ is a regular Sasakian structure 
then the space of leaves $X$ is a K\"ahler manifold, and we have a circle
bundle $S^1\hookrightarrow N \,\rightarrow\, M$ as above. 
If $N$ has a quasi-regular Sasakian structure, then the space of leaves $M$ is a K\"ahler
orbifold with cyclic quotient singularities, and there is an orbifold circle bundle
$S^1 \hookrightarrow N\,\rightarrow\, X$ such that the contact form $\eta$ satisfies $d\eta\,=\,\pi^*(\omega)$,
where $\omega$ is the orbifold K\"ahler form. {{Note that the map $\pi$ is an orbifold submersion, 
so that $\pi^*(\omega)$ is a well-defined (smooth) $2$-form on the total space $N$, which is a smooth
manifold. This defines a Sasakian structure on $N$ by \cite[Theorem 20]{MRT}.}}

\subsection{A simply connected non-formal Sasakian manifold} \label{subsec:7.1}

Examples of simply connected non-formal Sasakian manifolds, of dimension $2n+1\geq 7$, are given in \cite{BFMT}. There it is proved
that those examples are non-formal because they are not 3-formal, in the sense of Definition \ref{def:primera}. Here we 
show the non-formality proving that they have a non-trivial triple Massey product. 

Note that
if $N$ is a simply connected, compact and non-formal manifold (not necessarily 
Sasakian), then $\dim N\,\geq\, 7$.
Indeed, Theorem \ref{fm2:criterio2} gives that simply connected compact manifolds of dimension at most $6$
are formal \cite{FM, N-Miller}. Moreover, a $7$-dimensional
simply connected Sasakian manifold is formal if and only if all the triple Massey products are trivial \cite{MT}.
 
To construct a simply connected non-formal Sasakian $7$-manifold, 
we consider the K\"ahler manifold $M\,=\,S^2 \times S^2\times S^2$ with K\"ahler form
$$
\omega\,=\,\omega_1 +\omega_2 + \omega_3\, ,
$$
where $\omega_1$, $\omega_2$ and $\omega_3$ are the generators of the integral cohomology group of each of the 
$S^2$-factors on $S^2 \times S^2\times S^2$.
Let $N$ be the total space $N$ of the principal $S^1$-bundle
 $$
S^1 \,\hookrightarrow \,N \,\too\, M\,=\,S^2\times S^2\times S^2\, ,
 $$ 
with Euler class $[\omega]\in H^2(M,\mathbb{Z})$. Then, $N$ is a simply connected compact (regular) Sasakian manifold,
with contact form $\eta$ such that $d \eta \,=\, \pi^*(\omega)$.

From now on, we write $a_i\,=\,[\omega_i]\in H^2(S^2)$. Since $M\,=\,S^2\times S^2\times S^2$ is formal, a model of $M$ is
$(H^*(S^2\times S^2\times S^2), 0)$, where $H^*(S^2\times S^2\times S^2)$ is the de Rham 
cohomology algebra of $S^2\times S^2\times S^2$, that is
\begin{align*} \label{eqn:vic2}
 H^0(M)&=\langle 1\rangle, \nonumber\\
 H^1(M)&= H^3(M)\,=\,H^5(M)\,=\,0\,, \nonumber\\
 H^2(M)&= \langle a_1,\, a_2,\, a_3\rangle, \\
H^4(M)&=\langle a_1\cdot a_2,\, a_1\cdot a_3,\, a_2\cdot a_3\rangle, \nonumber\\
H^6(M)&=\langle a_1\cdot a_2\cdot a_3\rangle. \nonumber
\end{align*}
Therefore, a model of $N$ is the DGA $\Big(H^*(M)\otimes\bigwedge(x),\,d\Big)$, where $|x| \,=\,1$, 
$d(H^*(M))\,=\,0$ and $dx\,=\,a_1+a_2+a_3$. By Lemma \ref{lemm:massey-models}, we know that Massey products on a manifold can be computed by using any model for the manifold. 
Since $a_1\cdot a_1 \, = \, 0$ and $a_1\cdot a_2 \, = \, \frac{1}{2}d(a_1\,\cdot x +a_2\,\cdot x -a_3\,\cdot x)$,
we have that the (triple) Massey product $\langle a_1, a_1, a_2 \rangle\,=\,\frac{1}{2}[(a_1\,\cdot a_2-a_1\,\cdot a_3)\cdot x]$ is defined and it is non-trivial. {{Note that there is no indeterminacy of the Massey product, since
it lives in $a_1 \cdot H^3(N) + a_2 \cdot H^3(N)$, but $H^3(N)=0$, since by the
Gysin sequence, it equals the kernel of $[\omega]: H^2(M)\to H^4(M)$, which is an isomorphism.}}
So $N$ is non-formal.

The case $n\,>\,3$ is similar and it is deduced as follows. Consider $B\,=\,S^2\times \stackrel{(n)}{\ldots} \times S^2$. 
Let $a_1,\ldots,a_{n}\in H^2(B)$ be the cohomology classes given by each of the $S^2$-factors. Then the K\"ahler class
is given by $[\omega]\,=\,a_1+\cdots + a_{n}$. Consider the circle bundle 
$$
S^1\, \hookrightarrow \, N \,\too\, B
$$
with first Chern class equal to $[\omega]$.

Using again Lemma \ref{lemm:massey-models}, we know that Massey products on $N$ can be computed by using any model for $N$. 
Since $B$ is formal, a model of $B$ is the DGA $\Big(H^*(B),\, 0\Big)$. Thus,
a model of $N$ is the DGA $\Big(H^*(B)\otimes\bigwedge(x),\,d\Big)$, where $|x| \,=\,1$, 
$d(H^*(B))\,=\,0$ and $dx\,=\,a_1+a_2+\ldots +a_n$. 
Now, one can check that $a_1\cdot a_1 \, = \, 0$ and 
$$
a_1\cdot a_2\ldots a_{n-2}\cdot a_{n-1} \, = \, \frac{1}{2} d\Big((a_1\cdot a_2\ldots a_{n-2}+a_2\cdot a_3\ldots a_{n-2}\cdot a_{n-1}-
a_2\cdot a_3\ldots a_{n-2}\cdot a_n)\cdot x\Big).
$$
Thus the Massey product $\langle a_1,\, a_1,\, a_2\cdot a_3 \ldots a_{n-2}\cdot a_{n-1}\rangle$ is 
defined and a representative is $[(a_1\,\cdot a_2\ldots a_{n-2}\cdot a_{n-1} -a_1\,\cdot a_2\ldots 
a_{n-2}\cdot a_{n})\cdot x]$ which is non-trivial. Hence, we conclude that $N$ is non-formal.

\subsection{Simply connected formal Sasakian manifolds with $b_2\,\not=\,0$}

The most basic example of a simply connected compact regular Sasakian manifold
is the odd-dimensional sphere $S^{2n+1}$ considered as the total space of the Hopf fibration
$S^{2n+1} \,\hookrightarrow \, {\mathbb{CP}^n}$. It is well-known that $S^{2n+1}$ is formal.
In this section, we show examples of simply connected compact Sasakian manifolds,
with second Betti number $b_{2}\,\not=\,0$, which are formal. 

Note that Theorem \ref{fm2:criterio2} implies that any simply connected compact manifold (Sasakian or not) of dimension $\leq 7$
and with $b_2 \,\leq 1$, is formal. Examples of 7-dimensional simply connected compact Sasakian manifolds, with $b_2 \,\geq 2$,
which are formal are given in \cite{FIM}. 

To show examples of simply connected formal Sasakian manifolds, of dimension $\geq 9$ and with $b_2\,\not=\,0$, we consider the K\"ahler manifold 
$$
M\,=\,{\mathbb{CP}}{}^{n-1} \,\times S^2\, ,
$$
with K\"ahler form 
$$
\omega\,=\,\omega_1 +\omega_2\, ,
$$
where $\omega_1$ and $\omega_2$ are the generators of the integral cohomology group of ${\mathbb{CP}}{}^{n-1}$ and $S^2$, respectively. 
Let $N$ be the total space $N$ of the principal $S^1$-bundle
 $$
S^1 \,\hookrightarrow \,N \,\too\, M\,=\,{\mathbb{CP}}{}^{n-1} \,\times S^2\, ,
 $$ 
with Euler class $[\omega]\,\in\, H^2(M,\,\mathbb{Z})$. Then, $N$ is a simply connected compact (regular) Sasakian manifold, of dimension $2n+1$,
with contact form $\eta$ such that $d \eta \,=\, \pi^*(\omega)$.

\begin{proposition}\label{sasak-formal:b2=1}
The total space $N$ of the circle bundle $S^1\,\hookrightarrow\, N \,\longrightarrow\, M\,=\,{\mathbb{CP}}{}^{n-1} \,\times S^2$, with Euler class
 $[\omega]$, is a simply connected compact Sasakian manifold, with second Betti number $b_2\,=\,1$, which is formal.
\end{proposition}

\begin{proof}
Suppose $n\,\geq 4$. We will determine a minimal model of the $(2n+1)$-manifold $N$. 

Clearly $M\,=\,{\mathbb{CP}}{}^{n-1} \,\times S^2$ is formal because $M$ is K\"ahler. Hence, 
a (non-minimal) model of $M$ is the DGA $(H^*(M),0)$, where $H^*(M)$ is the de Rham cohomology algebra of $M$. Thus, a (non-minimal) model of 
$N$ is the differential algebra 
$(\mathcal{A}, d)$, where 
$$
\mathcal{A}\,=\,H^*(M)\otimes \bigwedge(x), \quad |x|\,=\,1, \quad d(H^*(M))\,=\,0, \quad dx\,=\,a_{1}\,+\, a_{2},
$$
where $a_{1}$ is the integral cohomology class defined by the K\"ahler form $\omega_{1}$ on 
${\mathbb{CP}}{}^{n-1}$, and $a_{2}$ is the integral cohomology class defined by the K\"ahler form $\omega_{2}$ on 
$S^2$. Then, the minimal model associated to this model of $N$ is
$$
({\SM}\, , \,D)\,=\,(\bigwedge(a, b, z)\, ,\, D),
$$
where $|a|\,=\,2$, $|b|\,=\,3$ and $|z|\,=\,2n-1$, 
while the differential $D$ is given
by $Da\,=\,Db\,=\,0$ and $Dz\,=\,a^n$. 
Therefore, we get 
$$
N^i\,=\,0,
$$
for $1\,\leq i\,\leq n$. Then, Theorem \ref{fm2:criterio2} implies that $N$ is formal because it is $n$-formal.
\end{proof}

\subsection{Non-formal quasi-regular Sasakian manifolds with $b_1=0$} \label{ex:quasi-regular}
The previous examples can be tweaked to obtain also examples of quasi-regular Sasakian manifolds $P$, where
the base of the (orbifold) circle bundle $S^1 \hookrightarrow P\to X$ is an honest orbifold K\"ahler manifold $X$. 
Obtaining simply connected manifolds $P$ in this way is a delicate matter, since the fundamental group of $P$ 
relates to the \emph{orbifold fundamental group} of $X$, and not its fundamental group (see \cite{Kol} and 
\cite{MRT} for discussions on these issues). Therefore we content ourselves with writing down examples with $H_1(P,\,\ZZ)\,=\,0$.

Consider a complex $3$-torus $T^3=\CC^3/\Gamma$, where $\Gamma$ is the discrete subgroup of $\CC^3$ consisting 
of the elements $(z_1, z_2, z_3)\in \CC^3$
whose components  $z_1, z_2$ and $z_3$ are Gaussian integers.  Now consider
the action  of the finite group $\ZZ_2$ on $\CC^3$ given by 
\begin{eqnarray*}
\varphi\colon \CC^3 & \to & \CC^3\\
 (z_1, z_2, z_3) &\mapsto & (-z_1, -z_2, -z_3),
\end{eqnarray*}
where $\varphi$ is the generator of $\ZZ_2$. This action satisf{}ies that $\varphi(z+z')=\varphi(z)+\varphi(z')$,
for $z, z' \in \CC^3$. Moreover, $\varphi(\Gamma)=\Gamma$. Therefore, $\varphi$
induces an action on $T^3=\CC^3/\Gamma$ with $2^6$ fixed points corresponding to 
$(z_1=u_1 + i\, u_2,  z_2=u_3 + i\, u_4, z_3=u_5 + i\, u_6)$ with $u_i=0, \frac{1}{2}$.
Thus, the quotient space 
$$X={T^3}/\ZZ_2$$
is a K\"ahler orbifold of (real) dimension $6$
with $2^6$ isolated orbifold singularities of order $2$. In fact, one can check that
the  standard complex structure $J$ on $T^3$ commutes with the  $\ZZ_2$-action, that is
$(\varphi_{*})_z\circ J_{z} = J_{\varphi(z)}\circ (\varphi_{*})_z$, for any point $z\in T^3$.
Moreover, the standard Hermitian metric and the K\"ahler form $\omega'$ on $T^3$ are $\ZZ_2$-invariant, and so they
induce an orbifold Hermitian metric and an orbifold K\"ahler form $\omega$ on $X$, respectively.

By \eqref{orbi-fcoh}, the cohomology of $X$ is given by $H^1(X,\ZZ)=H^1(T^3,\ZZ)^{\ZZ_2}=0$, hence
$b_1(X)=0$. Now consider the orbifold circle bundle 
\begin{equation*}\label{eqn:vic}
S^1\hookrightarrow P \stackrel{\pi}{\too} X,
 \end{equation*}
given by $c_1(P)=[\omega]$. We have the following:

\begin{proposition}
The manifold $P$ is a $7$-dimensional quasi-regular Sasakian manifold $N$ with $b_1=0$ which is non-formal.
\end{proposition}

\begin{proof}
 The total space of the orbifold circle bundle $P$ has a Sasakian structure with  contact form $\eta$ such that
 $d\eta=\pi^*(\omega)$, by
 \cite[Theorem 20]{MRT} (the proof of this result is given in the K-contact case but it works also for the Sasakian case). 
 The Leray spectral sequence gives that $b_1(P)=0$.

Let us see that $P$ is non-formal. 
First note that the cohomology of $T^3$ is the exterior algebra $\bigwedge^*(x_1,\ldots,x_6)$, with
$|x_i|=1, 1\leq i\leq 6$. Then $H^*(X)=\bigwedge^{\text{even}}(x_1,\ldots, x_6)$. Let $a_1=x_1x_2$, $a_2=x_3x_4$,
$a_3=x_5x_6$, so that $[\omega]=a_1+a_2+a_3$. As in Subsection \ref{subsec:7.1},
there is non-trivial (triple) Massey product in $P$. 
Indeed, $a_1\cdot a_1=0$ and $a_1\cdot a_2=\frac{1}{2}\,d\big((a_1+a_2-a_3)\cdot  \eta\big)$.
Then,
 $$
 \la a_1,a_1,a_2\ra\,=\,\frac{1}{2}[(a_1\cdot a_2-a_1\cdot a_3)\cdot \eta],
 $$
where $d\eta=\pi^*(\omega)$. So $P$ is non-formal.

There is a geometrical explanation of the above Massey product. 
 If $T=\CC/\ZZ^2$ is the $2$-torus, then the quotient 
$T/\ZZ_2 \cong S^2$, as a topological manifold. Thus 
 $$
 T^3/(\ZZ_2\x \ZZ_2\x \ZZ_2) = (T/\ZZ_2) \x (T/\ZZ_2) \x (T/\ZZ_2) \cong S^2\x S^2\x S^2 =M\, ,
$$
where each of the factors of $\ZZ_2\x \ZZ_2\x \ZZ_2$ acts on each of the three factors of
$T^3=T\x T\x T$, respectively, and $M$ is the 6-manifold of Subsection \ref{subsec:7.1}.
Therefore, the orbifold $X$ sits in the middle of two quotient maps
$$
T^3\to X=T^3/\ZZ_2\to M \cong {T^3}/(\ZZ_2\x \ZZ_2\x \ZZ_2).
 $$ 
Then there is a diagram
 \begin{eqnarray*} 
 S^1  & \hookrightarrow P \longrightarrow & X \\ 
  || \,\, & \downarrow & \, \, \downarrow \\
  S^1 &  \hookrightarrow N \longrightarrow & M
 \end{eqnarray*}
where $N$ is the 7-manifold of Subsection \ref{subsec:7.1}.
So, $P$ and $N$ are the same topological manifold. 
Then the non-zero Massey product
of $N$ produces the non-zero Massey product for $P$, giving the non-formality of $P$.
\end{proof}

\section*{Acknowledgements}

We are grateful to the referees for their helpful comments.
The first author is supported by a Post-Doc grant at Philipps-Universit\"at Marburg. 
The second author is supported by the J. C. Bose Fellowship. The third author is 
partially supported through Project MINECO (Spain) MTM2014-54804-P and
Basque Government Project IT1094-16. The fourth 
author is partially supported by Project MINECO (Spain) MTM2015-63612-P.

\end{document}